\documentclass[12pt]{article}

\usepackage{fullpage}
\usepackage{amssymb}
\usepackage{amsmath}
\usepackage{amsthm}
\usepackage{graphicx}
\usepackage{multirow}
\usepackage{epsfig}
\usepackage{epstopdf}
\usepackage{color}
\usepackage{booktabs}
\usepackage{hyperref}
\usepackage[capitalise]{cleveref}
\usepackage{verbatim}
\usepackage[backend=bibtex,giveninits,block=space,sortcites=true,maxbibnames=99]{biblatex}
\usepackage{doi}
\renewbibmacro{in:}{}
\DeclareFieldFormat[article]{volume}{\mkbibbold{#1}}
\DeclareFieldFormat[article]{number}{(#1)}
\DeclareFieldFormat{doi}{\href{https://dx.doi.org/#1}{\textsc{doi}}}
\AtEveryBibitem{\clearfield{issn}\clearfield{isbn}\clearlist{language}}
\renewbibmacro*{volume+number+eid}{%
  \printfield{volume}%
  \printfield{number}%
  \setunit{\bibeidpunct}%
  \printfield{eid}}
\def\abx@missing@entry#1{\abx@missing{#1??}}
\usepackage{authblk}

\newtheorem{thm}{Theorem}[section]
\newtheorem{prop}[thm]{Proposition}

\newtheorem{lemma}[thm]{Lemma}
\newtheorem{coro}[thm]{Corollary}

\theoremstyle{remark}

\newtheorem{rmk}[thm]{Remark}

\newtheorem{conj}[thm]{Conjecture}
\newtheorem{pro}[thm]{Problem}

\crefname{thm}{Theorem}{Theorems}
\crefname{prop}{Proposition}{Propositions}
\crefname{lemma}{Lemma}{Lemmas}
\crefname{coro}{Corollary}{Corollaries}
\crefname{rmk}{Remark}{Remarks}
\crefname{example}{Example}{Examples}
\crefname{conj}{Conjecture}{Conjectures}
\crefname{definition}{Definition}{Definitions}
\crefname{pro}{Problem}{Problems}

\newcommand{\naturals}{\mathbb{N}} % set of natural numbers
\newcommand{\reals}{\mathbb{R}} % set of real numbers
\newcommand{\rational}{\mathbb{Q}} % set of rational numbers

\newcommand{\prob}{\mathbb{P}} % probability
\newcommand{\expect}{\mathbb{E}} % expectation
\newcommand{\var}{\mathbf{Var}} % variance
\newcommand{\convd}{\xrightarrow{d}} % convergence in distribution
\newcommand{\normaldist}{\mathcal{N}} % normal distribution

\newcommand{\eps}{\varepsilon} % epsilon
\newcommand{\dif}{\mathrm{d}} % differential

\newcommand{\graphseq}{\mathcal{G}} % sequence of graphs
\newcommand{\eulergenus}{\gamma^{E}} % Euler-genus
\newcommand{\Eulergenus}{\mathcal{E}} % Euler-genus polynomial
\newcommand{\starladder}{\operatorname{SL}} % star-ladder graph
\newcommand{\halfladder}{\operatorname{HL}} % half-open ladder graph

\newcommand{\vertamal}{\circledast} % vertex-amalgamation
\newcommand{\proj}{\Pi} % projection of permutations to smaller element set
\newcommand{\rotsys}{\mathcal{R}} % rotation systems of a graph
\newcommand{\faceelem}{\mathbf{R}} % element in group algebra for faces in rotation systems
\newcommand{\cyc}{\operatorname{cyc}} % number of cycles in a permutation
\newcommand{\faceproj}{\Phi} % projection of face elements to face-counting polynomial
\newcommand{\signrotsys}{\mathcal{R}_{\pm}} % signed rotation systems of a graph
\newcommand{\swapping}{\operatorname{Sw}} % swapping of a graph with self-gluing

\newcommand{\tdef}[1]{\textcolor{blue}{\emph{#1}}}

\begin{document}

\title{Asymptotic normality of embedding distributions of some families of graphs}

\author{Yichao Chen\thanks{\href{mailto:chengraph@163.com}{chengraph@163.com}.  Partially supported by the NNSFC under Grant No. 12271392.}}
\affil{School of Mathematics, Suzhou University of Science and Technology, Suzhou, 215009, China}
\author{Wenjie Fang\thanks{\href{mailto:wenjie.fang@univ-eiffel.fr}{wenjie.fang@univ-eiffel}. Partially supported by ANR IsOMA (ANR-21-CE48-0007) and ANR CartesEtPlus (ANR-23-CE48-0018).}}
\affil{Univ Gustave Eiffel, CNRS, LIGM, F-77454 Marne-la-Vall\'ee, France}
\author{Zhicheng Gao\thanks{\href{mailto:zhichenggao@cunet.carleton.ca}{zhichenggao@cunet.carleton.ca}. Partially supported by Carleton University Development Grant (No.189035)}}
\affil{School of Mathematics and Statistics, Carleton University, Ottawa, Ontario, K1S5B6, Canada}
\author{Jinlian Zhang\thanks{\href{mailto:jinlian916@hnu.edu.cn}{jinlian916@hnu.edu.cn}. Partially supported by the NNSFC under Grant No. 12201198.}}
\affil{School of Mathematics and Statistics, Hunan University of Finance and Economics, Changsha, China}

\maketitle

\centerline{\small Mathematics Subject Classification: 05A15, 05A16, 05C10}

\begin{abstract}
  Computing the embedding distribution of a given graph is a fundamental question in topological graph theory. In this article, we extend our viewpoint to a sequence of graphs and consider their asymptotic embedding distributions, which are often the normal distribution. We establish the asymptotic normality of several families of graphs by using tools from analytic combinatorics and probability theory. We expect that these tools  can be used on other families of graphs to establish the asymptotic normality of their embedding distributions. Several open questions and conjectures are also raised in our investigation.

  \smallskip

\end{abstract}

%%%%%%%%%%%%%%%%%%%%%%%%%%%%%%%%%%%%%%%%%%%%%%%%%%%%%%%%%%%%%%%%%%%%%%%%%%%%%
%%%%%%%%%%%%%%%%%%%%%%%%%%%%%%%%%%%%%%%%%%%%%%%%%%%%%%%%%%%%%%%%%%%%%%%%%%%%%

\section{Introduction}

In topological graph theory, we are interested in cellular embeddings of graphs on surfaces of varying genus, orientable or not. For a generic connected graph, it may have cellular embeddings on several surfaces with different genera, and the number of such embeddings, counted up to homeomorphism, may vary depending on the surface. Thus comes the natural question of studying the distributions of genus for cellular embeddings of a given graph $G$ on orientable surfaces, or that of Euler-genus on general surfaces. We call such distributions the embedding distribution of $G$.

In the study of embedding distributions, there is a long-standing conjecture by Gross, Robbins and Tucker \cite{GroRobTuc89} stating that the genus distribution of every graph is log-concave. Although confirmed in many cases, such as in \cite{stahl-genus,LW07,CL2010,GMT14,GMTW16,GrossMTW16,GROSS2016499,Carr2025,Wang2021,CGMT20},  Mohar \cite{mohar-counter-example} recently constructed counter-examples to this conjecture.

Instead of looking at the embedding distributions of one single graph, we may also consider families of structurally related graphs. In this case, when the size of graphs tends to infinity, it is often the case that the embedding distribution converges to a normal distribution up to proper rescaling \cite{ZPC21}. However, asymptotic normality is not universal: for example, it cannot occur for families in which all embeddings of each graph have the same genus, because the corresponding genus distributions have zero variance \cite{unique-genus-graph,Sir91}. %unless structurally constrained \cite{unique-genus-graph,Sir91}. 
Furthermore, the mean and the variance of such embedding distributions often grows linearly to the size of the graphs for graphs that are sparse and with at least a linear number of independent cycles, whose random embedding thus has a linear number of faces with high probability. In this case, intuitively, we may consider such a phenomenon in the framework of the central limit theorem, as the genus or the Euler-genus depends on the local embedding of edges around each vertex, whose effects are not strongly correlated, since a pair of vertices does not share a face in general. However, such results are difficult to formulate and then to prove in generality. For a given graph family, usually we need to compute all their embedding distributions to show that they are asymptotically normal.

In this article, we establish the asymptotic normality of the embedding distributions of several graph families, while giving fundamental tools and frameworks that may be useful for other graph families. We propose three types of tools, with definition of related terms deferred to the corresponding sections:
\begin{itemize}
\item \textbf{Central limit theorem} (\Cref{sec:tree-like}): It is known from \cite{GF87,CG18} that the embedding distribution of a bar-amalgamation of two graphs $G, H$ is the sum of that of $G$ and $H$. Using the central limit theorem, we may show that the embedding distribution of certain families of tree-like graphs, which are constructed by successive bar-amalgamation, is asymptotically normal.
\item \textbf{Pertubation} (\Cref{sec:perturbation}): Intuitively, adding a small enough perturbation to a sequence of random variables does not change its asymptotic behavior. We quantify this idea and applied it to ring-like graphs and strictly monotone sequences of graphs. Both families comes from the bar-ring construction \cite{CG19}, whose effect on embedding distribution can be seen as a perturbation. We thus show that the embedding distribution of these families under certain conditions is asymptotically normal, while using only limited information about their component graphs.
\item \textbf{Analytic combinatorics} (\Cref{sec:double-edge-cycle,sec:h-families}): For graph families that are ``periodic'' in some sense, it is possible to compute explicitly the generating function of their embedding distributions, and then use standard tools in analytic combinatorics to show asymptotic normality (\Cref{thm:rational-asympt-normal}). We first showcase this approach in \Cref{sec:double-edge-cycle} to the double-edge cycle graphs. Then, similar to the production matrix approach in \cite{GKMT18,CG19}, we generalize the generating function approach to so-called $H$-linear and $H$-circular graphs, showing that their embedding genus generating functions are rational fractions while offering a method to compute them explicitly. Such an explicit expression then leads to asymptotic normality results.
\end{itemize}

The rest of the article is organized as follows. In \Cref{sec:prelim}, we lay out the necessary definitions. Then in \Cref{sec:tree-like} we establish asymptotic normality for tree-like graphs using the central limit theorem. In \Cref{sec:perturbation}, we show how to use a perturbative point of view to establish asymptotic normality for ring-like graphs and strictly monotone sequences of graphs. In \Cref{sec:double-edge-cycle}, we showcase on the special family of double edge cycle graphs how we may get directly the limits of embedding distributions using explicit expression of their generating functions. In \Cref{sec:h-families}, using group algebra, we show that $H$-linear and $H$-circular families have rational generating functions for embedding distributions, which allows us to implement the algorithm and compute such generating functions automatically. Such a tool allows the approach in \Cref{sec:double-edge-cycle} to be applied to more general cases, including the genus distribution of grid graphs studied by Khan, Poshni and Gross in \cite{grid-genus}. Along the way, we also propose several open problems and conjectures.

\section{Preliminary} \label{sec:prelim}

A (undirected) \tdef{graph} $G = (V, E)$ consists of a vertex set $V$ and an edge set $E$ which is a multiset of unordered pairs of elements in $V$. In the following, we consider only connected graphs unless otherwise specified. An \tdef{embedding} of graph $G$ into a surface $S$ is a drawing of the graph on the surface $S$ in such a way that its edges intersect only at their endpoints, defined up to continuous perturbation. In the following, we only consider surfaces that are connected and closed. The classification of surfaces tells us that any such surface $S$ is homeomorphic to one of the following cases:
\begin{itemize}
\item The \tdef{orientable surface} $O_k$ with $k \geq 0$, which is a sphere with $k$ handles attached, and $k$ is called its \tdef{genus};
\item The \tdef{non-orientable surface} $N_j$ with $j \geq 1$, which is a sphere with $j$ crosscaps attached, and $j$ is called its \tdef{crosscap-number}.
\end{itemize}
Given an embedding $M$ of a graph $G$ on the surface $S$, the connected components of its complement are called its \textit{faces}. In this article, we consider only \tdef{cellular embeddings}, meaning that all its faces are topological disks.

We denote by $[n]$ the set of integers from $1$ to $n$. Given a graph $G = (V, E)$, a \tdef{dart} is a pair $(u, e)$ with $u \in V$, $e \in E$, and $u$ adjacent to $e$. When $e$ is a loop, we distinguish the two ends of $e$, thus giving two darts. There are exactly $2|E|$ darts in $G$. A \tdef{rotation system} $(\sigma, \tau)$ of $G$ is a pair of permutations of all the darts such that all elements in the same cycle of $\sigma$ are darts containing the same vertex, and every cycle of $\tau$ are of length $2$ and contains darts with the same edge but different vertices. We note that $\tau$ is fixed. Rotation systems describes embeddings on orientable surfaces, with cycles in $\sigma$ describing how darts are clockwise ordered around each vertex, and cycles in $\tau$ about which darts are from the same edge. It is clear that cycles in $\sigma \tau$ correspond to faces, giving half of the darts in the same face in counterclockwise order. We denote by $\rotsys(G)$ the set of rotation systems of $G$.

For embeddings on general surfaces, they can be described combinatorially by a \tdef{signed rotation system} $(\sigma, \tau)$, which is the same as a rotation system but with $\sigma, \tau$ being sign-symmetric permutations over the darts and their formal opposites, that is, $\sigma(-(u, e)) = -\sigma((u, e))$, along with the same conditions as in rotation systems. Here, we note that there are several possibilities for $\tau$. We denote by $\signrotsys(G)$ the set of signed rotation systems of $G$. An edge is \tdef{twisted} if one of its corresponding cycles in $\tau$ contains both a dart and the formal opposite of the other dart. More detailed discussions about combinatorial representations of embeddings can be found in \cite[Chapter~3]{MT01}.  Sign-symmetric permutations of the darts and their formal opposites form a group called the \tdef{type-B Coxeter group} over the darts. More generally, the type-B Coxeter group over an element set $E$ can be regarded as the group of sign-symmetric permutations over the element set $E \cup \overline{E}$, where an element $\overline{e} \in \overline{E}$ is the formal opposite of $e \in E$.

Let $v(M)$ (resp. $e(M)$, $f(M)$) be the number of vertices (resp. edges, faces) of a graph embedding $M$ on a surface $S$. By Euler's formula, we have
\[
  v(M) - e(M) + f(M) = 2 - \eulergenus(S),
\]
where $\eulergenus(S)$ is the \tdef{Euler-genus} of the surface $S$. Specifically, $\eulergenus(S) = 2k$ if $S$ is homeomorphic to the orientable surface $O_k$, and $j$ if $S$ is homeomorphic to the non-orientable surface $N_j$.

As a graph may be embedded on several different surfaces, we are interested in the number of embeddings of different genus. Let $\gamma_i(G)$ denote the number of embeddings of graph $G$ on the orientable surface $O_i$ with genus $i$. The sequence $\gamma_0(G), \gamma_1(G), \gamma_2(G), \ldots$ is the \tdef{genus distribution} of graph $G$. The \tdef{genus polynomial} of graph $G$ is the polynomial
\[
  \Gamma_G(x)=\sum_{i \geq 0} \gamma_i(G) x^i.
\]
Similarly, the \tdef{crosscap-number polynomial} of graph $G$ is defined as
\[
  \tilde{\Gamma}_G(y) = \sum_{j \geq 1} \tilde{\gamma}_j(G) y^{j},
\]
where $\tilde{\gamma}_j(G)$ is the number of embeddings of graph $G$ on the non-orientable surface $N_j$ with crosscap-number $j$. The sequence $(\tilde{\gamma}_j(G))_{j \geq 1}$ is referred to as the \tdef{crosscap-number distribution} of graph $G$. We then define the \tdef{Euler-genus polynomial} $E_G(x)$ of $G$ by
\[
  E_G(x) = \Gamma_G(x^2) + \tilde{\Gamma}_G(x).
\]
The \tdef{Euler-genus distribution} of $G$ is thus the sequence
\[
  \gamma_0(G), \tilde{\gamma}_1(G), \gamma_1(G) + \tilde{\gamma}_2(G), \ldots, \tilde{\gamma}_{2j-1}(G), \gamma_{j}(G) + \tilde{\gamma}_{2j}(G), \ldots.
\]
We refer to both the genus distribution and the Euler-genus distribution of a graph $G$ as its \tdef{embedding distributions}.

For a formal power series $F(x)=\sum_{k\ge 0}f_kx^k$, we use $[x^k]F(x)$ to denote the coefficient of $x^k$ in $F(x)$, that is,  $[x^k]F(x)=f_k$.

We define the random variables $\gamma_G$ and $\eulergenus_G$ by
\[
\prob[\gamma_G = k] = \frac{[x^k]\Gamma_G(x)}{\Gamma_G(1)},\quad \prob[\eulergenus_G = k] = \frac{[x^k]E_G(x)}{E_G(1)}. 
\]
In other words, they stand for the probability 

Thus, the \tdef{probability generating function} of $\gamma_G$, denoted by $F_{\gamma_G}(x)$ is given by
\[
  F_{\gamma_G}(x) = \sum_{k \geq 0} \prob[\gamma_G = k] x^k=\frac{\Gamma_G(x)}{\Gamma_G(1)}.
\]

The \tdef{minimum genus} and \tdef{maximum genus} of a graph $G$, denoted by $\gamma_{\min}(G)$ and $\gamma_{\max}(G)$ are important parameters for graphs on their embeddings. They are given by
\[
  \gamma_{\min}(G) = \min \{i : \gamma_i(G) > 0\}, \quad \gamma_{\max}(G) = \max \{i : \gamma_i(G) > 0\}.
\]
Similarly, we define the \tdef{minimum Euler-genus} and \tdef{maximum Euler-genus} of the graph $G$, denoted by $\eulergenus_{\min}(G)$ and $\eulergenus_{\max}(G)$ respectively.

The \tdef{Betti number} of a graph $G$ is given by $\beta(G) = e(G) - v(G) + 1$. The Betti number gives the number of independent cycles of 
$G$, and each such cycle contributes an independent orientability choice in the signed‑rotation representation for embeddings. For a graph $G$, we have $\frac{\Gamma_{G}(1)}{E_{G}(1)} = 2^{-\beta(G)}$ \cite{CG18}. This implies that, when the Betti number $\beta(G)$ is large, there are far more non-orientable embeddings than orientable ones. Furthermore, it suggests that, when $G$ is large, the Euler-genus distribution of $G$ approaches normality if and only if the crosscap-number distribution does so \cite{ZPC21}.

For simplicity, given a sequence of graphs $\graphseq = G_1, G_2, \ldots$, we denote the genus polynomial (resp. Euler-genus polynomial) of the graph $G_n$ by $\Gamma_{n}(x)$ (resp. $E_{n}(x)$).

Given a sequence of real random variables $(X_n)_{n \geq 1}$, we say that $X_n$ \tdef{converges in distribution} to a random $X$ with continuous cumulative distribution function, denoted by $X_n \convd X$, if the cumulative distribution function of $X_n$ converges to that of $X$ when $n \to \infty$. A sequence of random variables $(X_n)_{n \geq 1}$ is \tdef{asymptotically normal} if $(X_n - \expect X_n) / \sqrt{\var X_n} \convd \normaldist(0, 1)$. Let $\graphseq = (G_n)_{n \geq 1}$ be a sequence of graphs. We say that the genus distribution of $\graphseq$ is \tdef{asymptotically normal} if the sequence of random variables $(\gamma_{G_n})_{n \geq 1}$ is asymptotically normal. It is equivalent to say that, with some mean $e_n \in \reals$ and variance $v_n \in (0,\infty)$, we have
\[
 \lim_{n\rightarrow \infty }  \left| \frac{1}{\Gamma_{G_n}(1)} \sum_{0 \leq i \leq e_n+x\sqrt{v_n}} \gamma_{i}(G_n) - \int_{-\infty}^x \frac{e^{-u^2 / 2}}{\sqrt{2\pi}} \dif u \right| = 0,~~\forall x\in \mathbb{R}.
\]
The same applies to the definitions of the corresponding Euler-genus distribution.

In the following sections, we study the asymptotic normality of embedding distributions of several different families of graphs, with methods that can be potentially generalized to become applicable to a larger range of graphs.

\section{Tree-like graphs and central limit theorem} \label{sec:tree-like}

The focus of this section is the asymptotic embedding distribution for a sequence of tree-like graphs obtained by an operation on graphs called ``bar-amalgamation''. Let $G, H$ be two graphs, and $u$ (resp. $v$) a vertex of $G$ (resp. $H$). The \textit{bar-amalgamation} $G \oplus_{uv} H$ is defined as the new graph obtained by linking $G$ and $H$ by the new edge $\{u, v\}$. Let $d_u$ (resp. $d_v$) be the degree of $u$ (resp. $v$) in $G$ (resp. $H$). The genus and Euler-genus polynomials of $G \oplus_{uv} H$ are known from \cite{GF87,CG18}.

\begin{thm}[{\cite[Theorem~5]{GF87}}, {\cite[Theorem~5.2]{CG18}}]\label{thm:bar}
  For two graphs $G$ and $H$, with $u$ (resp. $v$) a vertex of $G$ (resp. $H$). Then
  \[
    \Gamma_{G \oplus_{uv} H}(x) = d_u d_v \Gamma_G(x) \Gamma_H(x), \quad \Eulergenus_{G \oplus_{uv} H}(x) = d_u d_v \Eulergenus_G(x) \Eulergenus_H(x).
  \]
\end{thm}

We thus have the following corollary in terms of genus distribution.
\begin{coro} \label{coro:bar}
  For two graphs $G$ and $H$, with $u$ (resp. $v$) a vertex of $G$ (resp. $H$), we have
  \[
    \gamma_{G \oplus_{uv} H} = \gamma_G + \gamma_H, \quad \eulergenus_{G \oplus_{uv} H} = \eulergenus_G + \eulergenus_H.
  \]
  Here, $\gamma_G$ and $\gamma_H$ are considered as independent, and the same holds for $\eulergenus_G$ and $\eulergenus_H$.
\end{coro}
\begin{proof}
  We check easily that $d_u, d_v$ in \cref{thm:bar} disappear when normalized, meaning that the generating function of $\gamma_{G \oplus_{uv} H}$ (resp. $\eulergenus_{G \oplus_{uv} H}$) is the product of those of $\gamma_G$ and $\gamma_H$ (resp. $\eulergenus_G$ and $\eulergenus_H$).
\end{proof}

From \cref{coro:bar}, it is easy to see that the genus distribution of successive bar-amalgamation is the sum of the genus distributions of component graphs. This observation immediately opens up a way to show asymptotic normality of genus distribution of bar-amalgamation sequence by some central limit theorem. The following proposition is a simple version of central limit theorem derived from the famous Lindeberg condition. We refer readers to \cite[Chapter~4]{Petrov} for more general versions of the central limit theorem.

\begin{prop} \label{prop:zhang-4}
  Let $(\xi_k)_{k \geq 1}$ be a sequence of independent random variables with finite second moments. Let $B_n = \left(\sum_{k=1}^n \var \xi_k \right)^{1/2}$. Assume that, for a sequence of positive constants $(C_n)_{n \geq 1}$, we have $|\xi_k| \leq C_n$ for all $k \leq n$ and $C_n = o(B_n)$.
  % Then, it holds that
  % \begin{align*}
  %   \lim_{n\rightarrow \infty}\sup_{x\in \reals}\left|\prob(\frac{\sum_{k=1}^n (\xi_k-\expect \xi_k)}{B_n}\leq x)-\int_{-\infty}^x\frac{1}{\sqrt{2\pi}}e^{-\frac{t^2}{2}}\dif t\right|=0.
  % \end{align*}
  Then $\sum_{k=1}^n \xi_k$ is asymptotically normal with variance $B_n$.
\end{prop}
\begin{proof}
  By Theorem 1 and Remark 2 in \cite[Section 4, Chapter 3]{Shi}, we only need to check the Lindeberg condition: for every $\eps>0$,
\[
  \lim_{n \to \infty} \frac{1}{B_n^2} \sum_{k=1}^n \expect \left[ |\xi_k - \expect \xi_k|^2 I_{\left\{|\xi_k-\expect\xi_k|\geq \eps B_n\right\}} \right] = 0,
\]
where $I_A$ is the indicator function of the event $A$, which is $1$ when $A$ occurs and $0$ otherwise. We check that
\begin{align*}
 \frac{1}{B_n^2}\sum_{k=1}^n \expect \Big[|\xi_k-\expect \xi_k|^2I_{\left\{|\xi_k-\expect\xi_k|\geq \eps B_n\right\}}\Big] &\leq \frac{4C_n^2}{B_n^2}\sum_{k=1}^n\prob(|\xi_k-\expect\xi_k|\geq \eps B_n)
 \\ &\leq \frac{4C_n^2}{B_n^2}\sum_{k=1}^n\frac{\expect |\xi_k-\expect\xi_k|^2}{\eps^2B_n^2} = \frac{4C_n^2}{\eps^2B_n^2}.
\end{align*}
Thus, as $C_n = o(B_n)$, the Lindeberg condition holds.
\end{proof}

Let $(H_n)_{n \geq 1}$ be a sequence of connected graphs. A sequence of \tdef{tree-like graphs} $\graphseq = (G_n)_{n \geq 1}$ is obtained by successive bar-amalgamations on graphs $(H_n)_{n \geq 1}$. More precisely, the tree-like graphs $(G_n)_{n \geq 1}$ are obtained recursively by the following procedure:
\begin{itemize}
\item For $n=1$, $G_1=H_1$;
\item For $n > 1$, $G_n = G_{n-1} \oplus_{uv} H_n$ with $u$ (resp. $v$) a vertex of $G_{n-1}$ (resp. $H_n$).
\end{itemize}

\begin{thm} \label{thm:pp-1}
  Let $(G_n)_{n \geq 1}$ be a sequence of tree-like graphs constructed from some graph sequence $(H_n)_{n \geq 1}$ where each graph with at most $M$ edges for some fixed $M$, such that $\gamma_{\max}(H_n) > \gamma_{\min}(H_n)$ for infinitely many $n$. Then, the sequence of random variables $(\gamma_{G_n})_{n \geq 1}$ is asymptotically normal. If, instead, we have $\eulergenus_{\max}(H_n) > \eulergenus_{\min}(H_n)$ for infinitely many $n$, the sequence of random variables $(\eulergenus_{G_n})_{n \geq 1}$ is asymptotically normal.
 % with mean $e_n = \sum_{k=1}^n \expect \gamma_{G_k}$ and variance $v_n = \sum_{k = 1}^n \var \gamma_{G_k}$.
\end{thm}
\begin{proof}
  We only deal with the case of $\gamma_{G_n}$, as the case of $\eulergenus_{G_n}$ works the same. Up to extracting an infinite sub-sequence, we can strengthen the second condition to $\gamma_{max}(H_n) > \gamma_{min}(H_n)$ for all $n \geq 1$.  First, we observe by Euler's formula that, as $H_n$ has at most $M$ edges, its genus is bounded by $M/2$.

  We recall that $\gamma_G$ is the random variable of the genus of a uniformly random embedding of $G$. By the definition of tree-like graphs and \Cref{coro:bar}, we have $\gamma_{G_n} = \sum_{k = 1}^n \gamma_{H_k}$.  Take $B_n = \left(\sum_{k = 1}^n \var \gamma_{H_k}\right)^{1/2}$. We have $B_n^2= \Theta(n)$, as there are finitely many such $H_i$'s with a bounded number of edges, the non-zero variance on $\gamma_{H}$ is bounded on both sides. Taking $C_n = M/2$, we check that the conditions in \Cref{prop:zhang-4} are satisfied, meaning that $(\gamma_{G_n})_{n \geq 1}$ is asymptotically normal.
\end{proof}

A \tdef{cactus graph} (or simply \tdef{cactus}) is a connected graph with maximum genus $0$. Thus, cacti are graphs in which any two cycles have no vertex in common. We note that this definition is different from the usual one where two cycles are allowed to share one vertex, but the current definition is motivated by the result in \cite{unique-genus-graph} stating that graphs with unique genus are exactly cactus graphs defined here.

\begin{rmk} \label{rmk:genus-discrepany-struct}
  It is known that, for a connected graph $H$, its maximum genus (resp. Euler-genus) equals to its minimum genus (resp. Euler-genus) if and only if $H$ is a cactus graph (resp. a tree graph) \cite{unique-genus-graph,Sir91,GT87} We can thus replace the genus discrepancy condition in \cref{thm:pp-1} by a condition on the graph structure. 
\end{rmk}

With a better lower bound on the variance of $\gamma_G$ or $\eulergenus_G$, we can strengthen \Cref{thm:pp-1}, as long as the conditions in \Cref{prop:zhang-4} are satisfied. We show an example of such strengthening below. We recall that the \tdef{girth} of a graph $H$ is the length of the shortest cycle in $H$. In particular, the girth of a graph without cycle is infinite. For the case of Euler-genus, we may then replace the hard limit on the number of edges in \Cref{thm:pp-1} by a soft one along with an extra condition on the girth.

\begin{thm} \label{thm:pp-girth}
  Let $(G_n)_{n \geq 1}$ be a sequence of tree-like graphs constructed from some graph sequence $(H_n)_{n \geq 1}$ of graphs of girth at most $k$ and maximal degree of vertices at most $\Delta$, such that for each $n$, $H_n$ has at most $f(n) n^{1/2}$ edges for a fixed function $f(n)$ with $f(n) = o(1)$ when $n \to \infty$. Then, the sequence of random variables $(\eulergenus_{G_n})_{n \geq 1}$ is asymptotically normal.
\end{thm}
\begin{proof}
  We first show that there is some constant $\nu > 0$ such that $\var \gamma_{H_n}^E \geq \nu$. As the girth of $H_n$ is at most $k$, there is a cycle $C$ formed by edges $e_1, \ldots, e_s$ with $e_i$ linking vertices $v_i$ and $v_{i+1}$, with $v_{s+1} = v_1$, and all vertices involved are distinct. As the graph $H_n$ is labeled implicitly on the vertices, we pick $C$ with minimal labels, and we take $e_1$ to be the edge with minimal labels. Then, the cycle $C$ forms a face boundary walk when all edges are not twisted, and the related darts are adjacent in the same direction. Let $d_i$ be the degree of $v_i$. We denote by $\mathrm{CF}$ the set of such embeddings. A random embedding is in $\mathrm{CF}$ with probability
  \[
    2^{-s} \prod_{i = 1}^s (2d_i - 2)^{-1} \geq (4 \Delta)^{-k}.
  \]
  For such an embedding $E$, twisting $e_1$ leads to an embedding $E'$ not in $\mathrm{CF}$, and this edge-twisting operation is reversible. Moreover, the Euler-genus of $E'$ is that of $E$ plus $1$, as there is one less face in $E'$ than in $E$. We denote by $\mathrm{CF'}$ the set of such embeddings obtained from those in $\mathrm{CF}$. The variance of Euler-genus in $\mathrm{CF} \cup \mathrm{CF'}$ is thus at least $1/4$, meaning that
  \[
    \var \eulergenus_{H_n} \geq 2 \cdot (4 \Delta)^{-k} \cdot \frac{1}{4}.
  \]
  We thus have the existence of the constant $\nu$.

  By \Cref{coro:bar} and the definition of tree-like graphs, we have $\gamma_{G_n}^E = \sum_{k=1}^{n} \gamma_{H_k}^E$. Now take $B_n = \left(\sum_{k = 1}^n \var \gamma_{H_k}^E \right)^{1/2}$. We have $B_n^2 = \Omega(n)$. Take $C_n$ to be the number of edges in $H_n$. We see that the Euler-genus of $H_n$ is bounded by $C_n$. As $C_n = f(n) n^{1/2} = o(B_n)$, the conditions in \Cref{prop:zhang-4} are thus satisfied, and we have the asymptotic normality of $\gamma_{G_n}^E$.
\end{proof}

%%%%%%%%%%%%%%%%%%%%%%%%%%%%%%%%%%%%%%%%%%%%%%%%%%%%%%%%%%%%%%%%%%
%%%%%%%%%%%%%%%%%%%%%%%%%%%%%%%%%%%%%%%%%%%%%%%%%%%%%%%%%%%%%%%%%%
\section{Asymptotic normality via perturbation} \label{sec:perturbation}

In this section, we will show the asymptotic normality of genus and Euler-genus distributions for several graph families by observing that such distributions can be split into two parts: a ``major part'' that can be showed easily to converge to the normal distribution, and a ``perturbative part'' that can be ignored. The families we will discuss all come from a construction called \emph{bar-ring}. We now first introduce the related concepts.

A vertex $u$ of a graph $G$ is a \tdef{pendant vertex} if its degree is $1$. We denote by $(G, u, v)$ the graph with two marked pendant vertices $u$ and $v$. We now introduce two \tdef{partial genus distributions} for $(G, u, v)$. For simplicity, sometimes we omit the two marked pendant vertices and write only $G$ when it is clear in the context.

\begin{description}
\item[$\mathbf{d}_G(j)$]: the number of embeddings of $(G, u, v)$ of genus $j$ such that the faces incident to $u$ and $v$ are all distinct.
\item[$\mathbf{s}_{G}(j)$]: the number of embeddings of $(G,u,v)$ of genus $j$ in which the pendant vertices $u$ and $v$ lie on the boundary of the same face.
\end{description}

We define two \tdef{partial genus polynomials} of $(G,u,v)$, denoted by $D_G(x)$ and $S_G(x)$, as
\begin{equation}
  \label{eq:partial-euler-genus-poly}
  D_{G}(x) = \sum_{j \geq 0} \mathbf{d}_{G}(j) x^j, \quad S_{G}(x) = \sum_{j \geq 0} \mathbf{s}_{G}(j)x^j.
\end{equation}

Similarly, by replacing genus with Euler-genus, we define the \tdef{partial Euler-genus distributions} given by $\tilde{\mathbf{d}}_G(j)$ and $\tilde{\mathbf{d}}_G(j)$ and the corresponding \tdef{partial Euler-genus polynomials} $\tilde{D}_G(x)$ and $\tilde{S}_G(x)$.

A \tdef{bar-ring} $G^\circ$ of $n$ graphs $G_1, \ldots, G_n$, with each graph $G_i$ having two pendant vertices $u_i$ and $v_i$, is obtained by adding an edge between $v_i$ and $u_j$ for all $1 \leq i \leq n$ and $j \equiv i + 1 \pmod{n}$ in the disjoint union of $G_1, \ldots, G_n$.  Figure~\ref{fig} shows a bar-ring of four graphs $G_1,G_2,G_3,G_4$.

\begin{figure}
\centering
\includegraphics[width=2.5in]{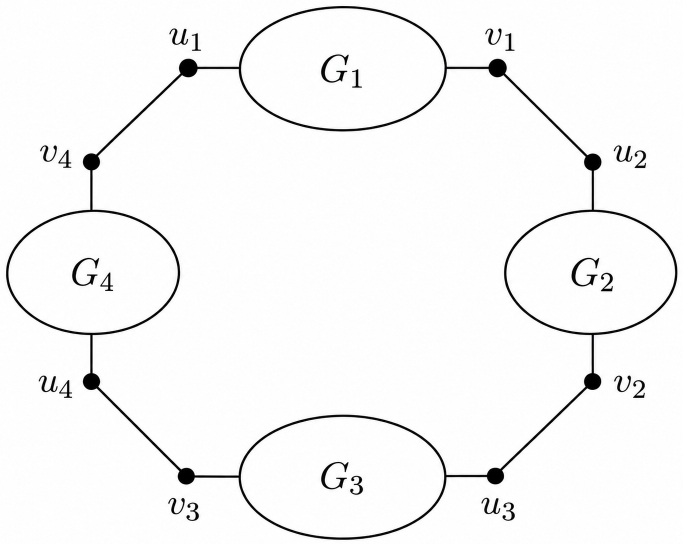}
\caption{A bar-ring of four graphs $G_1,G_2,G_3,G_4$.}
\label{fig}
\end{figure}

\begin{lemma}[{\cite[Theorem~5.3]{CG19}}]\label{lem:bar-ring-genus}
  The genus polynomial of a bar-ring $G^\circ$ of some graphs $G_1, \ldots, G_n$ is given by
\[
  \Gamma_{G^\circ_n} = x \prod_{i=1}^n \Gamma_{G_i}(x) + (1 - x) \prod_{i=1}^n S_{G_i}(x).
\]
\end{lemma}

\begin{lemma} \label{lem:bar-ring-euler-genus}
  The Euler-genus polynomial of the bar-ring $G^\circ$ of some graphs $G_1, \ldots, G_n$ is given by
  \[
    E_{G^\circ}(x) = 2x^2 \prod_{i=1}^n E_{G_i}(x) + (1+x-2x^2) \prod_{i=1}^{n} \tilde{S}_{G_i}(x).
  \]
\end{lemma}
\begin{proof}
The result follows directly from Theorem~5.3 and Theorem~6.4 in \cite{CG19}.
\end{proof}

We observe that \Cref{lem:bar-ring-genus,lem:bar-ring-euler-genus} have similar form on the right-hand side as sum of two products, with the terms in the first product ``containing'' those in the second. Therefore, we are tempted to show that the first term is dominant, while the second is perturbative. Then to show that the first term indeed leads to a convergence to the normal distribution, we can either perform explicit computation or use the central limit theorem as in \Cref{sec:tree-like}.

In order to deal with perturbation at the level of the probability generating function, which are normalized genus and Euler-genus polynomials in our case, we propose the following lemma. Given a real-valued polynomial $P(x)$, we define $\Sigma(P)$ to be the sum of the absolute values of all its coefficients. It is clear that $\Sigma(P) \geq |P(1)|$.

\begin{lemma} \label{lem:gen-fct-perturbe}
  Let $(P_n(x))_{n \geq 1}$ be a sequence of polynomials with positive coefficients such that $P_n(1) \to \infty$ when $n \to \infty$, and $(Q_n(x))_{n \geq 1}$ a sequence of polynomials with real coefficients such that $P_n(x) + Q_n(x)$ has all its coefficients positive for all $n$ and $Q_n(1) \geq 0$. Let $(X_n)_{n \geq 1}$ (resp. $(X'_n)_{n \geq 1}$) be the sequence of independent random variables where $X_i$ with values in $\naturals$ is defined by its probability generating function $P_n(x) / P_n(1)$ (resp. $(P_n(x) + Q_n(x)) / (P_n(1) + Q_n(1))$). Suppose that $(X_n)_{n \geq 1}$ is asymptotically normal, and $\Sigma(Q_n) = o(P_n(1))$. Then $(X'_n)_{n \geq 1}$ is also asymptotically normal.
\end{lemma}
\begin{proof}
  We first show that the $L^1$-norm between $X_n$ and $X'_n$ tends to $0$ when $n \to \infty$. We have
  \begin{align*}
    \Vert X_n - X'_n \Vert_{1}
    &= \sum_{k \geq 0} \left| \frac{[x^k]P_n(x) + [x^k]Q_n(x)}{P_n(1) + Q_n(1)} - \frac{[x^k]P_n(x)}{P_n(1)} \right| \\
    &\leq \sum_{k \geq 0} \left| \frac{[x^k]Q_n(x)}{P_n(1) + Q_n(1)} \right| + \sum_{k \geq 0} \left| \frac{Q_n(1) [x^k]P_n(x)}{P_n(1) (P_n(1) + Q_n(1))} \right| \\
    &= \frac{\Sigma(Q_n) + Q_n(1)}{P_n(1) + Q_n(1)} = o(1).
  \end{align*}
  The last equality is due to the fact that $\Sigma(Q_n(1)) = o(P_n(1))$ and $\Sigma(Q_n) \geq |Q_n(1)|$. We then conclude as the convergence in $L^1$-norm is stronger than convergence in distribution.
\end{proof}

\subsection{Ring-like graph families} \label{sec:ring-like}

In this subsection, we consider ring-like graphs, which are bar-rings with graphs of arbitrary sizes. We first consider ring-like graphs formed by a special family of sub-graphs called \emph{half-open ladders} as components, in which case we can perform explicit computation for bounds needed in \Cref{lem:gen-fct-perturbe}. We then also consider the more general case where we relax the constraints on the components of bar-rings.

 Let $P_n$ be the path graph with $n \geq 1$ vertices. An \tdef{$n$-rung ladder}, denoted as $L_n$, is the Cartesian product of $P_n$ with the complete graph $K_2$. An \tdef{$n$-rung half-open ladder}, denoted as $\halfladder_n$, is obtained from $L_n$ by attaching a pendant edge to each of the two vertices of an end rung.  %An \tdef{$n$-rung half-open ladder}, denoted as $\halfladder_n$, is the Cartesian product of $P_{n + 1}$ with the complete graph $K_2$ with the edge from $K_2$ at one end removed. 
 For instance, $\halfladder_1$ is a path of length $3$ for the special case $n = 1$. The edges from $K_2$ are called \tdef{rungs}, and it is clear that $\halfladder_n$ always has two pendant vertices. For a tuple of non-negative integers $\alpha = (\alpha_1, \ldots, \alpha_k)$, the \tdef{star-ladder} with signature $\alpha$, denoted by $\starladder_\alpha$, is the bar-ring of $\halfladder_{\alpha_1}, \ldots, \halfladder_{\alpha_k}$.

The genus and Euler-genus polynomials of star-ladders with given signature are known \cite{CG19,CGM12}. In the following, we perform an explicit computation for the genus polynomial of star-ladders.

\begin{prop}
  \label{prop:half-ladder-partial-genus}
  For $n \geq 0$, The half-open ladder $\halfladder_n$ has the partial genus polynomials
  \[
    D_{\halfladder_n}(x) = \sum_{k=0}^{\lfloor n/2 \rfloor} \binom{n-k}{k} 2^{n+k+1} x^k, \quad S_{\halfladder_n}(x) = \sum_{k=0}^{\lfloor (n + 1) / 2\rfloor}
    \binom{n + 1 - k}{k} 2^k x^k.
  \]
  Furthermore, we have $3D_{\halfladder_n}(1) \geq S_{\halfladder_n}$
\end{prop}
\begin{proof} We notice that $\halfladder_{n}$ is obtained by removing one of the end-rungs of $L_{n+1}$. Let $e$ be an end-rung of $L_n$, we define the following two partial genus distributions of $L_n$:
  \begin{description}
  \item[$d_n(i)$]: the number of embeddings of $L_n$ on $O_i$ such that $e$ is adjacent to two different faces;
  \item[$s_n(i)$]: the number of embeddings of $L_n$ on $O_i$ such that $e$ is adjacent to only one face.
  \end{description}

  Let $D_n(x) = \sum_{i \geq 0} d_n(i) x^i$ and $S_n(x) = \sum_{i \geq 0} s_n(i) x^i.$ We now find closed formulas for $D_n(x)$ and $S_n(x)$. Using the techniques of counting genus distributions of the cobblestone path graphs \cite{FGS89}, we have the following recursive formulas for the partial genus distributions of the ladder graph $L_n$:
  \[
    d_n(i) = 2 d_{n-1}(i) + 4 s_{n-1}(i), \quad s_{n}(i) = 2 d_{n-1}(i-1).
  \]
  Rewriting the recurrence relations above in terms of generating functions, we obtain
  \[
    D_n(x) = 2 D_{n-1}(x) + 4 S_{n-1}(x), \quad S_{n}(x) = 2 x D_{n-1}(x).
  \]
  A simple substitution leads to
  \[
    D_n(x) = 2 D_{n-1}(x) + 8 x D_{n-2}(x),
  \]
  with initial conditions $D_0(x) = 1$ and $D_1(x) = 2$. Let $F_D(t, x) = \sum_{t \geq 0} t^n D_n(x)$. We have
  \[
    F_D(t, x) - 1 - 2t = 2t(F_D(t, x) - 1) + 8 t^2 x F_D(t, x).
  \]
  Solving the equation above, we have $F_D(t, x) = (1 - 2t - 8t^2x)^{-1}$. By extracting first the coefficient of $x^k$ then that of $t^n$ from $F_D(t, x)$, we get an explicit formula of $D_n(x)$:
  \[
    D_n(x) = \sum_{k=0}^{\lfloor n/2 \rfloor} \binom{n-k}{k} 2^{n+k} x^k.
  \]
  As $S_n(x) = 2x D_{n-1}(x)$, we also have
  \[
    S_n(x) = \sum_{k=0}^{\lfloor (n-1)/2\rfloor} \binom{n-k-1}{k} 2^{n+k} x^{k+1}.
  \]

  \begin{figure}
    \centering
    \includegraphics[width=1.5in]{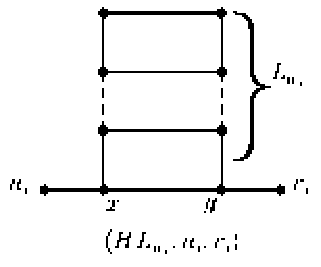}
    \caption{The graph $(\halfladder_{n_i},u_i,v_i)$}\label{fig:Ladder-pend}
  \end{figure}

  We observe that the $n$-rung half-open ladder $\halfladder_n$ can be obtained by connect two isolated vertices $u$ and $v$ to the two degree 2 vertices $x$ and $y$ of an end-rung of $L_{n}$ respectively (see \Cref{fig:Ladder-pend}). By face-tracing \cite{GT87}, the two partial genus polynomials for the structure $(\halfladder_{n}, u_i, v_i)$ are expressed $D_{\halfladder_n}(x) = 2D_{n}(x)$ and $S_{\halfladder_n}(x) = 4S_{n}(x) + 2D_{n}(x)$. A computation gives our first claim. Our second claim comes naturally from the fact that $D_n(1) \geq S_n(1)$.
\end{proof}

\begin{prop} \label{prop:half-ladder-partial-euler-genus}
  For the partial Euler-genus polynomials $\tilde{D}_{\halfladder_n}(x)$ and $\tilde{S}_{\halfladder_n}(x)$ of the half-open ladder $\halfladder_n$, we have $7 \tilde{D}_{\halfladder_n}(1) \geq \tilde{S}_{\halfladder_n}(1)$.
\end{prop}
\begin{proof}
  As in \Cref{prop:half-ladder-partial-genus}, we partition the embeddings of the $n$-rung ladder $L_n$ on surfaces with Euler-genus $i$ into two \tdef{partial Euler-genus distributions} as follows:

  \begin{description}
  \item[$\tilde{d}_{n}(i)$]: the number of embeddings of $L_n$ on surfaces of Euler genus $i$ such that the end-rung lies on the boundaries of two different faces.
  \item[$\tilde{s}_{n}(i)$]: the number of embeddings of $L_n$ on surfaces of Euler-genus $i$ such that the end-rung lies twice on the boundary of the same face.
  \end{description}

  The two \tdef{partial Euler-genus polynomials} of $L_n$ are again denoted by
  $\tilde{D}_{n}(x) = \sum_{i \geq 0} \tilde{d}_{n}(i) x^i$ and $\tilde{S}_{n}(x) = \sum_{i \geq 0} \tilde{s}_{n}(i) x^i$. From \cite{CG18}, we have
  \begin{align*}
    \tilde{D}_{n}(x) &= 2 \tilde{D}_{n-1}(x) + 4 \tilde{S}_{n-1}(x), \\
    \tilde{S}_{n}(x) &= (2x + 4x^2) D_{n-1}(x) + 4 x S_{n-1}(x),
  \end{align*}
  with initial conditions $\tilde{D}_{0}(x) = 1$ and $\tilde{S}_{0}(x) = x$. The specialization $x = 1$ thus gives
  \[
    \tilde{D}_n(1) = 2 \tilde{D}_{n-1}(1) + 4 \tilde{S}_{n-1}(1), \quad \tilde{S}_{n}(1) = 6 \tilde{D}_{n-1}(x) + 4 \tilde{S}_{n-1}(1),
  \]
  along with the initial conditions $\tilde{D}_0(1) = \tilde{S}_0(1) = 1$. Solving the recurrence gives
  \[
    \tilde{D}_n(1) = \frac{1}{5} (4 \cdot 8^n + (-2)^n), \quad \tilde{S}_n(1) = \frac{1}{5} (6 \cdot 8^n - (-2)^n).
  \]
  Using the very crude bound $8^n \geq (-2)^n$ for any $n \geq 0$, we have $3 \tilde{D}_n(1) \geq \tilde{S}_n(1)$.

  Again, the two partial genus polynomials for the structure $(\halfladder_{n}, u_i, v_i)$ are given by $\tilde{D}_{\halfladder_n}(x) = 2\tilde{D}_{n}(x)$ and $\tilde{S}_{\halfladder_n}(x) = 4\tilde{S}_{n}(x) + 2\tilde{D}_{n}(x)$. Thus, we have $7 \tilde{D}_{\halfladder_n}(1) \geq \tilde{S}_{\halfladder_n}(x)$
\end{proof}
\begin{thm} \label{thm:star-ladder-limit}
  Let $\alpha = (\alpha_1, \alpha_2, \ldots)$ be an infinite sequence of strictly positive integers, and $\alpha|_k = (\alpha_1, \ldots, \alpha_k)$ the tuple of its first $k$ entries. Then both the genus and the Euler-genus distribution of the star-ladder $\starladder_{\alpha|_k}$ is asymptotically normal as $k$ tends to infinity.
\end{thm}
\begin{proof}
  Below, we prove the case for the genus distribution. The case of the Euler-genus distribution follows the same approach by replacing \cref{lem:bar-ring-genus} by \cref{lem:bar-ring-euler-genus}, and \cref{prop:half-ladder-partial-genus} by \Cref{prop:half-ladder-partial-euler-genus} with some changes in constants.

  The definition of star ladders and Lemma~\ref{lem:bar-ring-genus} implies that
  \[
    \Gamma_{\starladder_{\alpha|_k}} = x \prod_{i = 1}^k \Gamma_{\halfladder_{\alpha_i}}(x) + (1 - x) \prod_{i = 1}^k S_{\halfladder_{\alpha_i}}(x).
  \]
  As every $S_{\halfladder_n}(x)$ has only positive coefficients, we have
  \[
    \Sigma\left( (1 - x) \prod_{i=1}^k S_{\halfladder_{n_i}}(x) \right) \leq 2 \Sigma\left( \prod_{i=1}^k S_{\halfladder_{n_i}}(x) \right) = 2 \prod_{i=1}^n S_{\halfladder_{n_i}}(1).
  \]

  By \Cref{prop:half-ladder-partial-genus}, we have
  \[
    \frac{S_{\halfladder_n}(1)}{\Gamma_{\halfladder_n}(1)} = \frac{S_{\halfladder_n}(1)}{D_{\halfladder_n}(1) + S_{\halfladder_n}(1)} \leq \frac{3}{4},
  \]
  When $k \to \infty$, it is clear that $\prod_{i=1}^k S_{\halfladder_{\alpha_i}}(1) = o(\prod_{i=1}^k \Gamma_{\halfladder_{\alpha_i}}(1))$. By \Cref{lem:gen-fct-perturbe}, we have the asymptotic normality of the genus distribution of $(\starladder_{\alpha|_k})_{k \geq 2}$.
\end{proof}

\begin{thm} \label{thm:general-bar-ring}
  Let $\graphseq = (G_i)_{i \geq 1}$ be a sequence of graphs, and $G^\circ_n$ a bar-ring formed by $G_1, \ldots, G_n$ in this order. Suppose that
 \begin{itemize}
 \item Every $G_i$ is connected, and there are infinitely many $i$'s such that $\gamma_{\max}(G_i) > \gamma_{\min}(G_i)$;
 \item There is some constant $M$ such that every $G_i$ has at most $M$ edges.
 \end{itemize}
 Then the embedding distributions of $(G^\circ_n)_{n \geq 1}$ are asymptotically normal.
\end{thm}
\begin{proof}
  Up to extraction of an infinite sub-sequence, we assume that $\gamma_{\max}(G_i) > \gamma_{\min}(G_i)$ for all $i$. For simplicity, we identify vertices in $G^\circ_n$ with the ones in the corresponding $G_i$'s. For each $G_i$, we define $\overline{G_i}$ to be its 2-core, that is, the subgraph obtained by removing pendant vertices successively until none exists. The vertices removed form \emph{tree-components} of $G_i$. As we may attach tree-components arbitrarily in an embedding of $\overline{G_i}$ without altering its genus, clearly $\gamma_{G_i} = \gamma_{\overline{G_i}}$, meaning that $\gamma_{\max}(\overline{G_i}) > \gamma_{\min}(\overline{G_i})$. Let $u$ (resp. $v$) be the attaching point that remains in $\overline{G_i}$ of the tree-component of $G$ containing $u_i$ (resp. $v_i$). Note that it is possible that $u = v$. By the property of 2-core, there is a closed walk $C$ that goes through both $u$ and $v$. Note that $C$ may contain the same edge twice, but passes the same vertex at most twice.

  We note that the reasoning above holds for the graphs in $\graphseq$, not depending on orientability of embeddings. We now first consider the case of genus distribution. To establish a lower bound for $D_{G_i}(1)$, we notice that, if there is an embedding $M$ of $\overline{G_i}$ such that a directed run of $C$ is the border of a face, then by attaching the tree-component containing $u_i$ and $v_i$ to different sides of $C$, we have an embedding of $G_i$ that is counted by $\mathbf{d}_{G_i}(1)$. For a directed run of $C$ to be a border of a face in $M$, we now show that for each vertex $w \in C$, we need to fix $2$ edges in the cyclic order around $w$. The case when $C$ passes by $w$ only once is clear. Otherwise, $w$ is a cut vertex in $\overline{G_i}$, and we only need to fix the order of the two adjacent edges in one of the passes of $C$. For the attachment of tree-components containing $u_i$ and $v_i$, we simply fix them just after the edges of $C$ through $u$ and $v$ respectively. We thus have
  \[
    \frac{D_{G_i}(1)}{\Gamma_{G_i}(1)} \geq \left( \prod_{w \in G_i} d_w (d_w - 1) \prod_{w \in G_i, d_w \geq 3} (d_w - 2) \right)^{-1} > 0.
  \]
  As every $G_i$ has at most $M$ edges, and there is finitely many such graphs, there is a constant $c > 0$ such that $D_{G_i}(1) / \Gamma_{G_i}(1) \geq c$, meaning that $S_{G_i}(1) / \Gamma_{G_i}(1) \leq 1 - c$. We thus have $\prod_{i=1}^n S_{G_i}(1) = o(\prod_{i=1}^n \Gamma_{G_1}(1))$ when $n \to \infty$. By \Cref{lem:bar-ring-genus,lem:gen-fct-perturbe}, we know that the genus distributions of $(G^\circ_n)_{n \geq 1}$ is asymptotically normal.

  For the case of Euler-genus distribution, we notice that the construction above also works to give a lower bound on $\tilde{D}_{G_i}(1) / \Eulergenus_{G_i}(1)$, which suffices for an application of \Cref{lem:bar-ring-euler-genus,lem:gen-fct-perturbe} to obtain asymptotic normality.
\end{proof}

\begin{rmk} \label{rmk:bar-ring-strengthening}
  In the same vein of the discussion above \Cref{thm:pp-girth}, we may strengthen \Cref{thm:general-bar-ring} if we can bound $\frac{D_{G_i}(1)}{\Gamma_{G_i}(1)}$ above a strictly positive constant. For instance, this is the case when all the $G_i$'s have uniformly bounded maximal degrees, and the minimal lengths of cycles passing through both $u, v$ are also bounded uniformly. We see that this case is reminiscent to the conditions in \Cref{thm:pp-girth}. This is the case for star-ladders, which means that \Cref{thm:star-ladder-limit} is a special case of this strengthening.

  We note that in the proof of \Cref{thm:general-bar-ring}, it does not suffice to show that $D_{G_i}(1) / \Gamma_{G_i}(1) > 0$. For instance, if $D_{G_i}(1) / \Gamma_{G_i}(1) = (i + 1)^{-2}$, then $\left( \prod_{i=1}^n S_{G_i}(1) \right) / \left( \prod_{i=1}^n \Gamma_{G_1}(1)) \right) = \frac{n + 2}{2(n + 1)} \to 1/2$, for which \Cref{lem:gen-fct-perturbe} does not apply. The precise condition needed here for $\prod_{i=1}^n S_{G_i}(1) = o(\prod_{i=1}^n \Gamma_{G_1}(1))$ is that the infinite sum $\sum_{i \geq 1} \left( D_{G_i}(1) / \Gamma_{G_i}(1) \right)$ converges.
\end{rmk}

%%%%%%%%%%%%%%%%%%%%%%%%%%%%%%%%%%%%%%%%%%%%%%%%%%%%%%%%%%%%%%%%%%%%%%%%%%%%%
%%%%%%%%%%%%%%%%%%%%%%%%%%%%%%%%%%%%%%%%%%%%%%%%%%%%%%%%%%%%%%%%%%%%%%%%%%%%%

\subsection{Strictly monotone sequences of graphs} \label{sec:strictly-monotone}

In this subsection, we consider sequences of graphs that are constructed by adding extra structures to edges using a extremely restricted kind of bar-ring construction, but such construction can be repeated in an arbitrary number of times. However, in this case, under proper assumption, we may still use the perturbative point of view to obtain asymptotical normality.

Let $e = uv$ be an edge of $G$, we attach an \tdef{open ear} to $e$ by replacing the edge $e$ with a 3-path $uxyv$ (or a 3-cycle $(uxy)$ if $u = v$) and adding a multiple edge paralleling the new edge $xy$. Alternatively, we attach a \tdef{closed ear} to $e$ by replacing $e$ with a 2-path $uxv$ (or a 2-cycle $(ux)$ if $u = v$) and attaching a loop to the new vertex $x$. We may attach several open and closed ears serially to $e$, one followed by another. For example, Figure~\ref{fig:e} shows the result of attaching two open and two closed ears to an edge $e = uv$. %For example, let $K_2$ represent the complete graph comprising two vertices, $u$ and $v$. We subsequently attach two open ears and two closed ears in sequence to the edge $e = uv$, as depicted in Figure \ref{fig:e}.

\begin{figure}
  \centering
  \includegraphics[width=10cm]{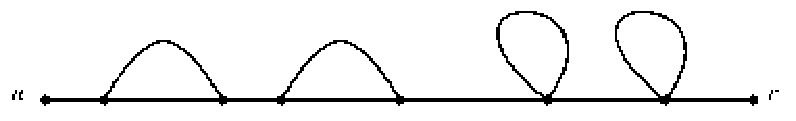}
  \caption{Attaching two open ears and two closed ears  to $K_2$.} \label{fig:e}
\end{figure}

A sequence of graphs $\graphseq = (G_i)_{i \geq 1}$ is \tdef{strictly monotone} if the following is satisfied:
\begin{enumerate}
\item For any $i, j$, the graphs $G_i$ and $G_j$, are homeomorphic, \emph{i.e.}, can be obtained one from another by edge subdivision and smoothing, if and only if $i=j$.
\item Each $G_i$ is homeomorphic to a subgraph of $G_{i+1}$.
\end{enumerate}
For instance, the sequences of complete graphs $(K_n)_{n \geq 1}$ and complete bipartite graphs $(K_{n, m})_{n \geq 1}$ are strictly monotone. The concept of strictly monotone sequences of graphs was first introduced by Chen and Gross in \cite{CG92b}. Using the theory of bridges and attachment as presented by Tutte \cite{Tut84}, they showed that attaching ears serially is the only method to construct limit points of values of average genus for 2-connected graphs. In a subsequent article \cite{Che10}, these findings were further expanded to encompass any type of graph by eliminating the requirement of 2-connectivity.

\begin{figure}
\centering
\includegraphics[width=9cm]{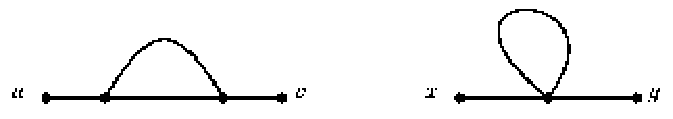}
\caption{Two graphs $(A,u,v)$ and $(B,x,y)$}. \label{fig:AB}
\end{figure}

Given a graph $G$, let $e$ be an edge of $G$. We denote by $G_{r,s}$ the graph obtained by attaching attaching $r$ open ears and $s$ closed ears serially in arbitrary order to $e$. We denote by $G'$ the graph obtained by breaking $e$ in $G$ into two edges, each adjacent to one of the pendant vertices $u', v'$. Let $(A, u, v)$ and $(B, x, y)$ be the graphs given in \Cref{fig:AB}. We see that $G_{r,s}$ is homeomorphic to a bar-ring of a sequence of $r$ copies of $(A, u, v)$, $s$ copies of $(B, x, y)$, and then $(G', u', v')$. We now show that the Euler-genus distribution of $G_{r,s}$ is asymptotically normal.

\begin{prop}\label{prop:gpoly:N}
  Let $G_{r,s}$ be a graph obtained by attaching $r$ open ears and $s$ closed ears serially to an edge of $G$, and $G'$ the graph obtained from breaking $e$ into two edges to pendant vertices in $G$. Then the Euler-genus distribution of $G_{r,s}$ is asymptotically normal when $r + s \to \infty$, with mean $(r + s)/2 + 1 + \expect \eulergenus_{G'}$ and variance $(r + s)/4 + \var \eulergenus_{G'}$.
\end{prop}
\begin{proof}
  From \Cref{lem:bar-ring-euler-genus} and the fact that $D_A(x) = D_B = 2$,  $S_A(x) = 2 + 4x$, and $S_{B}(x) = 4 + 6x$, we have the following expression of $\Eulergenus_{G_{r,s}}(x)$:
  \begin{align*}
    \Eulergenus_{G_{r,s}}(x) &= P(x) + Q(x), \\
    P(x) &= 2 x^2 \Eulergenus_{G'}(x) (4+4x)^r (6+6x)^s, \\
    Q(x) &= (1 + x - 2x^2) S_{G'}(x) (2+4x)^r (4+6x)^s.
  \end{align*}

  It is clear by the usual central limit theorem that the random variable $X$ corresponding to $P(x)$ is asymptotically normal with mean $(r + s)/2 + 1 + \expect \eulergenus_{G'}$ and variance $(r + s)/4 + \var \eulergenus_{G'}$ asymptotically, as $X$ is the sum of $r + s$ independent random Bernoulli variables with parameter $1/2$ along with $\eulergenus_{G'} + 2$. We have $P(1) = 2 E_{G'}(1)8^r 12^s$. Now, recall that, given a polynomial $R(x)$, we denote by $\Sigma(R)$ the sum of the absolute values of all its coefficients. We have $\Sigma(Q) \leq 4 S_{G'}(1) 6^r 10^s = o(P(1))$ when $r + s \to \infty$. By \Cref{lem:gen-fct-perturbe}, we know that the random variable $\eulergenus_{G_{r,s}}$ given by the generating function $\Eulergenus_{G_{r,s}}(x)/\Eulergenus_{G_{r,s}}(1)$ is asymptotically normal when $r + s \to \infty$.
\end{proof}

\begin{thm}\label{thm:Asynor:bicon}
  Let $\graphseq = (G_n)_{n \geq 1}$ be a strictly monotone sequence of 2-edge-connected graphs such that the maximum genera of all graphs $G_j$ are bounded by some uniform constant. Then the Euler-genus distributions of $G_{n}$ are asymptotically normal.
\end{thm}
\begin{proof}
  If $H$ is homeomorphic to a subgraph of $G$, then $\gamma_{\max}(H) \leq \gamma_{\max}(G)$. Hence, the sequence $(\gamma_{\max}(G_n))_{n \geq 1}$ is increasing. As it is also bounded, $\gamma_{\max}(G_n)$ becomes stationary. By \cite[Theorem~1.2]{CL07} or \cite[Theorem~4.5]{Jia94}, there exists an index $N$ such that all but finitely many graphs in $\graphseq$ can be obtained by attaching ears serially to edges $e \in E$ for some subset $E$ of edges in $G_N$.

  For $n > N$, we assume that $G_n$ is obtained by attaching $r_{e, n}$ open ears and $s_{e, n}$ closed ears sequentially to each edge $e \in E$ of $G_N$. Up to taking some larger $N$ and modifying $E$, we may assume that $r_{e, n} + s_{e, n} \to \infty$ when $n \to \infty$. We take $r_n = \sum_{e \in E} r_{e, n}$, $s_n = \sum_{e \in E} s_{e, n}$, and $k = |E|$.

  Using \Cref{lem:bar-ring-genus} successively on each $e \in E$, the Euler-genus polynomial of $G_n$ can be expressed as
  \begin{align*}
    \Eulergenus_{G_n}(x) &= P_n(x) + Q_n(x), \\
    P_n(x) &= 2^{k} x^{2k} \Eulergenus_{G_N}(x) (4 + 4x)^{r_n} (6 + 6x)^{s_n},
  \end{align*}
  and $Q_n(x)$ is a sum of terms that are products of some constant polynomial with terms of the form $(4 + 4x)^{r_{e, n}}$, $(6 + 6x)^{s_{e, n}}$, $(2 + 4x)^{r_{e, n}}$, and $(4 + 6x)^{s_{e, n}}$. Furthermore, each summand in $Q_n(x)$ contains at least one of the two last types of terms. We thus conclude by an argument similar to that in the proof of \Cref{prop:gpoly:N} that $\Sigma(Q_n) = o(P_n(1))$, as $r_{e, n} + s_{e, n} \to \infty$ when $n \to \infty$. It is thus clear that the random variables $X_n$ associated with $P_n(x)$ are asymptotically normal, with mean $(r_n+s_n)/2 + 2k + \expect \eulergenus_{G_N}$, and variance $(r_n+s_n)/4 + \var \eulergenus_{G_N}$. By \Cref{lem:gen-fct-perturbe}, we know that the distribution of $\eulergenus_{G_n}$ is asymptotically normal.
\end{proof}

\begin{thm} \label{thm:Asynor:max-genus}
  Let $\graphseq = (G_n)_{n \geq 1}$ be a strictly monotone sequence of graphs whose maximum orientable genus is uniformly bounded from above. Then the Euler-genus distributions of $\graphseq$ are asymptotical normal.
\end{thm}
\begin{proof}
 Similar to the argument at the beginning of the proof of \Cref{thm:Asynor:bicon}, there is an index $N$ such that all but finitely many graphs in $\graphseq$ can be obtained from $G_N$ by attaching ears serially to edges in a subset $E\subseteq E(G_N)$ and by attaching cacti through bar-amalgamation after subdividing edges in $E$. For $n \geq N$, we denote by $r_{e, n}$ and $s_{e, n}$ the number of open (resp. closed) ears added serially to $e \in E$ to obtain $G_n$, and $t_n$ the sum of Betti numbers of each cactus added to edges in $E$ in $G_N$ by bar-amalgamation after subdividing the edges to obtain $G_n$. Again, up to changing $N$, we may assume that for each $e \in E$, we have $r_{e, n} + s_{e, n} \to \infty$ when $n \to \infty$. We take $r_n = \sum_{e \in E} r_{e, n}$, $s_n = \sum_{e \in E} s_{e, n}$, and $k = |E|$.

  We first observe that the Euler-genus polynomial of a cactus $C$ with vertex set $V(C)$ and edge set $E(C)$ is
  \[
    \Eulergenus_C(x) = 2^{|E(C)| - \beta(C)} (1 + x)^{\beta(C)} \prod_{v \in V(C)} (d_v - 1)!.
  \]
  This is because $\beta(C)$ is the number of cycles in a cactus, and for an embedding of $C$, its genus depends on the number of cycles that are not separating on the surface, \emph{i.e.}, having an odd number of twisted edges. Each such cycle contributes $1$ to the Euler-genus of the embedding. We thus have the expression by accounting for all possibilities.

  By \Cref{lem:bar-ring-euler-genus}, and using the same reasoning as in the proof of \Cref{thm:Asynor:bicon}, the Euler-genus polynomial of $G_n$ can be expressed as
  \begin{align*}
    \Eulergenus_{G_n} &= P_n(x) + Q_n(x), \\
    P_n(x) &= \kappa 2^k x^{2k} (1 + x)^{t_n} \Eulergenus_{G_N}(x) (4 + 4x)^{r_n} (6 + 6x)^{s_n},
  \end{align*}
  where $\kappa$ is a constant depending only on the attached cacti, and $Q(x)$ satisfies $\Sigma(Q_n) = o(P_n(1))$ when $r_{e, n} + s_{e, n} + t_{e, n} \to \infty$. The random variable $X_n$ associated to $P_n(x)$ is the sum of $2k$, $\eulergenus_{G_N}$ and $r_n + s_n + t_n$ independent random variables following the Bernoulli distribution with parameter $1/2$. It is clear that $(X_n)_{n \geq N}$ is asymptotically normal, with mean $2k + \expect \eulergenus_{G_N} + (r_n + s_n + t_n) / 2$ and variance $\var \eulergenus_{G_N} + (r_n + s_n + t_n) / 4$. By \Cref{lem:gen-fct-perturbe}, we know that $(\eulergenus_{G_n})_{n \geq N}$ is also asymptotically normal, with the same mean and variance.
\end{proof}

%%%%%%%%%%%%%%%%%%%%%%%%%%%%%%%%%%%%%%%%%%%%%%%%%%%%%%%%%%%%%%%%%%%%%%%%%%%%%
%%%%%%%%%%%%%%%%%%%%%%%%%%%%%%%%%%%%%%%%%%%%%%%%%%%%%%%%%%%%%%%%%%%%%%%%%%%%%

\section{Double-edge cycle graphs via analytic combinatorics} \label{sec:double-edge-cycle}

For some particular families of graphs, we may compute explicitly the generating functions of their embedding distributions. In this case, we may use tools in analytic combinatorics to obtain the asymptotic normality of these embedding distributions. In this section, we apply this approach to double-edge cycle graphs. Our method relies crucially on the following theorem, which is a special case of a more general theorem in \cite{flajolet}.

\begin{thm}[{\cite[Theorem~IX.9]{flajolet}}, adapted to rational generating functions] \label{thm:rational-asympt-normal}
  For $F(t, x) = A(t, x) / B(t, x)$ a combinatorial generating function, \emph{i.e.}, all coefficients of $F(t, x)$ are positive integers, suppose that $F(t, x)$ satisfies the following conditions:
  \begin{enumerate}
  \item $F(t, x)$ is analytic at $(0, 0)$;
  \item $F(t, 1)$ is analytic at $t = 0$, whose dominant singularity $\rho^*$ (\emph{i.e.}, with minimum modulus) ensured by Pringsheim's theorem \cite[Theorem~IV.6]{flajolet} is a simple zero of $B(t, 1)$, and $A(\rho^*, 1) \neq 0$;
  \item $\partial_t B(\rho, 1) \cdot \partial_x B(\rho, 1) \neq 0$, which ensures a non-constant $\rho(x)$ analytic at $x = 1$ such that $B(\rho(x), x) = 0$ and $\rho(1) = \rho^*$;
  \item $\mathfrak{v} \neq 0$, where
    \[
      \mathfrak{v} = \frac{\rho'(1)^2 - \rho(1) \rho''(1) - \rho(1) \rho'(1)}{\rho(1)^2} \neq 0.
    \]
  \end{enumerate}
  Then the random variable $X_n$ with probability generating function
  \[
    p_n(x) = \frac{[t^n]F(t, x)}{[t^n]F(t, 1)}
  \]
  is asymptotically normal with mean $\mathfrak{m}n + O(1)$ and variance $\mathfrak{v}n + O(1)$, where $\mathfrak{m} = -\frac{\rho'(1)}{\rho(1)}$.
\end{thm}

For its usage, given a graph family $(G_n)_{n \geq 1}$, we define the generating function $\Gamma(t, x) = \sum_{n \geq 1} \Gamma_{G_n}(x) t^n$ over their genus polynomials. If we can show that $\Gamma(t, x)$ is rational and satisfies all the conditions in \Cref{thm:rational-asympt-normal}, then we have the asymptotic normality of the genus distribution of $G_n$, with explicit formulas for its mean and variance. The same also works for Euler-genus distribution. We now show that such an approach leads to the asymptotic normality of both the genus and the Euler-genus distributions of double-edge cycle graphs.

\subsection{Genus distribution for double-edge cycle graphs}

Let $C_n$ denote the cycle graph with $n$ vertices. When each edge of $C_n$ is replaced with a pair of parallel edges, we get the \tdef{double-edge cycle graph}, denoted by $C_n^2$. The double-edge cycle graph has been widely used as a fundamental graph for constructing quadrangular embeddings of the complete graph \cite{HartsfieldRingel1989O,HartsfieldRingel1989N,LiuEtAl2019,Nog16}. We have the following results about the genus distribution of $C_n^2$. Note that the Euler's formula implies that the genus $g_n$ of $C_n^2$ satisfies $2 g_n \le n$, thus the following results essentially cover the whole range of $g$.

\begin{thm} \label{thm:double-cycle-genus}
  The genus distribution of $C_n^2$ is asymptotically normal with mean $(1/4) n + O(1)$ and variance $(3/32) n + O(1)$. Furthermore, for $\tau = \frac{1}{1 - 2g/n}$, as $n \to \infty$ and uniformly for $g/n$ in any closed subinterval of $(0, 1/2)$, we have
  \[
    \gamma_g(C_n^2) \sim \frac{\tau}{3} \sqrt{\frac{2\tau(\tau^2 - 1)}{\pi n}} \left(\frac{3}{\tau^2 - 1}\right)^{g}(2(1 + \tau))^{n}.
  \]
\end{thm}
\begin{proof}
  The genus polynomial $\Gamma_n(x)$ for $C_n^2$ has been calculated by Gross~\cite{Gr11}, Baek and Park~\cite{BP11}, and Chen and Gross~\cite{CG18}. The recurrence below is derived in~\cite{CG18}:
  \begin{equation}
    \Gamma_n(x) = \sum_{j=1}^6 b_j(x) \Gamma_{n-j}(x), \label{eq:recursion1}
  \end{equation}
  where
  \begin{align*}
    &b_1(x) = 6, &&b_{2}(x) = 28x - 8, &&&b_{3}(x) = -96x,\\
    &b_4(x) = -16x (15x-2), &&b_5(x) = 288x^2, &&&b_6(x) = 576x^3,
  \end{align*}
  with initial values
  \begin{align*}
    &\Gamma_1(x) = 2x + 4, &&\Gamma_2(x) = 30x + 6,\\
    &\Gamma_3(x) = 72x^2 + 136x + 8, &&\Gamma_4(x) = 840x^2 + 440x + 16,\\
    &\Gamma_5(x) = 1440x^3 + 4832x^2 + 1472x + 32, &&\Gamma_6(x) = 18912x^3 + 22496x^2 + 5184x + 64.
  \end{align*}
  We define the generating function
  \[
    \Gamma(t, x) = \sum_{n \ge 1} \Gamma_n(x) t^n.
  \]
  It follows from \eqref{eq:recursion1} and the initial values that $\Gamma(t, x) = A(t, x) / B(t, x)$ is rational, where
  \begin{align}
    \begin{split}
      A(t, x) &= \sum_{j=1}^6 \Gamma_j(x) t^j - \sum_{j=1}^5 \sum_{k=1}^{6-j} b_j(x) \Gamma_k(x) t^{j+k} \\
             &= 288x^2 (x+1) t^6 - 48x (2x^2 - 7x - 3) t^5 - 16 (15x^2 - 5x - 1) t^4 \\
             &\quad\quad\quad + 4 (4x^2 -35x + 1) t^3 + 18 (x - 1) t^2 + 2 (x + 2) t,
    \end{split}
    \\
    B(t, x) &= 1 - \sum_{j=1}^6 b_j(x) t^j = (1 - 4xt^2) (1 - 4t - 12xt^2) (1 - 2t - 12xt^2).
  \end{align}
  Thus, the singularities of $\Gamma(t, x)$ (as a function of $t$) are the zeros of $B(t, x)$, which are
  \[
    t = \frac{\pm\sqrt{1+3x}-1}{6x}, \quad t = \frac{\pm\sqrt{1+12x}-1}{12x}, \quad t = \pm\frac{1}{2\sqrt{x}}.
  \]
  The dominant singularity is 
  \begin{equation}
    \rho(x) := \frac{\sqrt{1+3x}-1}{6x} = \frac{1}{2(1+\sqrt{1+3x})}, \label{eq:r}
  \end{equation}
  which is a solution of $1 - 4t - 12xt^2 = 0$ for $t$. As $\rho(1) = 1/6$, it is easy to check that $A(\rho(1), 1) = 32/81 \ne 0$, $\rho(1)$ is a simple zero of $B(t, 1)$, and that $\partial_tB(\rho(1), 1) \cdot \partial_xB(\rho(1), 1) = 512/2187 \neq 0$. It is also clear that $\Gamma(t, x)$ is analytic at $(0, 0)$, while $\Gamma(t, 1)$ is analytic around $t = 0$. By explicit computation, \Cref{thm:rational-asympt-normal} applies and we have our claim.

  For the ``Furthermore'' part, we note that $\rho(x)$ satisfies $|\rho(x)| > \rho(|x|)$ when $x \ne |x|$. Applying \cite[Theorem~3]{Ben73}, when $n \to \infty$, we have the following uniformly for all $g / n$ in a closed sub-interval of $(0, 1/2)$:
  \begin{equation} \label{eq:local-limit}
    [x^g] \Gamma_n(x) \sim \frac{x}{\mathfrak{v}(x) \sqrt{2\pi n}} x^{-g} \rho(x)^{-n},
  \end{equation}
  where $x>0$ satisfies $\mu(x) = g / n$ and $\mathfrak{v}(x)$ defined by
  \begin{equation} \label{eq:var}
    \mathfrak{v}(x) = \frac{3x}{4}(1 + 3x)^{-3/2}.
  \end{equation}
  Note that $\mathfrak{v}(1)$ is equal to the variance coefficient $\mathfrak{v}$ in \Cref{thm:rational-asympt-normal}. Setting $\tau = \sqrt{1 + 3x}$ in \Cref{eq:local-limit} and using \Cref{eq:var}, we complete the proof.
\end{proof}

\subsection{Euler-genus distribution for double-edge cycle graphs} \label{sec:double-edge-euler-genus}

Let $P_n$ be the path graph with $n$ vertices, and $P_n^2$ the \tdef{double-edge path graph} obtained by doubling each edge of $P_n$. Let $u, v$ be the degree-two vertices of $P_n^2$. Then, we observe that $C_n^2$ is obtained from $P_n^2$ by connecting $u$ and $v$ with a pair of parallel edges. Our procedure for calculating the Euler-genus polynomials of $C_n^2$ is as follows. First, we distinguish ten cases for embeddings of $P_n^2$ on Euler-genus $i$ surfaces. Then, using face tracing, we establish recurrences on the ten partial Euler-genus polynomials for $P_n^2$. Finally, using the edge-adding rules, we obtain a recurrence on the Euler-genus polynomial of $C_n^2$ by adding a pair of parallel edges from $P_n^2$. %The detailed calculation procedure for the Euler-genus distribution of $C_n^2$ can be found \Cref{apdx:double-cycle-euler-genus}.

\subsubsection{Partial Euler-genus types and recurrence relations}

Let $G(s,t)$ denote a connected graph with two vertices $s$ and $t$, each pf degree $2$. Let $S_i$ represent a surface of Euler-genus $i$. The standard procedure for calculating the Euler-genus distribution of $G(s, t)$ divides the embeddings of $G(s, t)$ on $S_i$ into 10 general types, similar to the (partial) genus distribution introduced by Gross \cite{Gr11}. For each $i \geq 0$, the quantities below count the embeddings of $G(s,t)$ on $S_i$ satisfying the following conditions:

\begin{itemize}
\item Type $dd_i^0(G)$: two distinct faces are incident to $s$ and two
distinct faces are incident to $t$, and all four faces are distinct.
\item Type $ds_i^0(G)$: two distinct faces are incident to $s$ and one
face is incident to $t$, and all three faces are distinct.
\item Type $sd_i^0(G)$: one face is incident to $s$ and two distinct
faces are incident to $t$, and all three faces are distinct.
\item Type $ss_i^0(G)$: one face is incident to $s$ and another with
$t$, and these two faces are distinct.
\item Type $dd_i'(G)$: two distinct faces are incident to each of $s$
and $t$, and exactly one face is incident to both $s$ and $t$.
\item Type $dd_i''(G)$: two distinct faces are incident to each of $s$
and $t$, and both faces are incident to both $s$ and $t$.
\item Type $ds_i'(G)$: two distinct faces are incident to $s$ and one
face is incident to $t$, and the face incident to $t$ is also incident
with $s$.
\item Type $sd_i'(G)$: one face is incident to $s$ and two distinct
faces are incident to $t$, and the face incident to $s$ is also
incident to $t$.
\item Type $ss_i^1(G)$: the same face is incident to both $s$ and $t$,
with face-boundary pattern $sstt$.
\item Type $ss_i^2(G)$: the same face is incident to both $s$ and $t$,
with face-boundary pattern $stst$.
\end{itemize}

For $i \geq 0$, it is evident that the number of embeddings $\eulergenus_i(G)$ of $G$ on $S_i$ is the sum of its ten partial genus distributions. We define the vector
\[
  \left(
    \begin{array}{cccccccccc}
      dd_i^0 & ds_i^0 & sd_i^0 & ss_i^0 & dd_i' & dd_i'' & ds_i' & sd_i' & ss_i^1 & ss_i^2 \\
    \end{array}
  \right)
\]
as the \emph{partial Euler-genus distribution vector} (or \emph{pEd-vector} for short) of $G$.

Suppose that $H(u, v)$ is a connected graph with two degree $2$ vertices $u$ and $v$. By identifying the vertices $t$ of $G(s, t)$ and $u$ of $H(u, v)$, we obtain a new graph $W(s, v)$ also denoted by $W(s, v) = G(s, t) \ast H(u, v)$. The embeddings of $W(s, v)$ are also divided into ten types. \Cref{tab:mult1,tab:mult2} explain in detail how to obtain the embedded types of $W(s, v)$ from the ten embedded types of $G(s, t)$ and $H(u, v)$. It is important to note that each embedding of both $G(s, t)$ and $H(u, v)$ produce six distinct embeddings of $W(s,v)$. The data in \Cref{tab:mult1,tab:mult2} are obtained using the face-tracing algorithm, applied $600 = 6 \times 10 \times 10$ times.

Now, let each $G(u_i, v_i)$ be a copy of $H(u, v)$. An \tdef{$H(u,v)$-linear graph}, denoted as $G_n(u_1,v_n)$, is constructed from a sequence of $n$ graphs: $G(u_1, v_1), G(u_2, v_2), \ldots, G(u_n, v_n)$ by pasting vertices $v_j$ and $u_{j+1}$ for all $j$ with $1 \leq j \leq n - 1$.

\begin{table}[!p]
  \centering
  \begin{tabular}{ccl}
    \toprule
    $G(u, v)$ & $H(s, t)$ & $G(s, t) \ast H(u, v)$ \\
    \midrule
    $dd_i^0$ & \multirow{10}{*}{$dd_j^0$} & $4dd_{i+j}^0 + 2dd_{i+j+2}^0 $ \\
    $ds_i^0$ & & $6dd_{i+j}^0$ \\
    $sd_i^0$ & & $4sd_{i+j}^0 + 2sd_{i+j+2}^0$ \\
    $ss_i^0$ & & $6sd_{i+j}^0$ \\
    $dd_i'$ & & $4dd_{i+j}^0 + 2dd_{i+j+2}^0$ \\
    $dd_i''$ & & $4dd_{i+j}^0 + 2dd_{i+j+2}^0$ \\
    $ds_i'$ & & $6dd_{i+j}^0$ \\
    $sd_i'$ & & $4sd_{i+j}^0 + 2sd_{i+j+2}^0$ \\
    $ss_i^1$ & & $6sd_{i+j}^0$ \\
    $ss_i^2$ & & $4sd_{i+j}^0 + 2dd_{i+j}^0$ \\
    \midrule
    $dd_i^0$ & \multirow{10}{*}{$ds^0$} & $4ds_{i+j}^0 + 2ds_{i+j+2}^0 $ \\
    $ds_i^0$ & & $6ds_{i+j}^0$ \\
    $sd_i^0$ & & $4ss_{i+j}^0 + 2ss_{i+j+2}^0$ \\
    $ss_i^0$ & & $6ss_{i+j}^0$ \\
    $dd_i'$ & & $4ds_{i+j}^0 + 2ds_{i+j+2}^0$ \\
    $dd_i''$ & & $4ds_{i+j}^0 + 2ds_{i+j+2}^0$ \\
    $ds_i'$ & & $6ds_{i+j}^0$ \\
    $sd_i'$ & & $4ss_{i+j}^0 + 2ss_{i+j+2}^0$ \\
    $ss_i^1$ & & $6ss_{i+j}^0$ \\
    $ss_i^2$ & & $4ss_{i+j}^0 + 2ds_{i+j}^0$ \\
    \midrule
    $dd_i^0$ & \multirow{10}{*}{$sd^0$} & $6dd_{i+j}^0$ \\
    $ds_i^0$ & & $6dd_{i+j}^0$ \\
    $sd_i^0$ & & $6sd_{i+j}^0$ \\
    $ss_i^0$ & & $6sd_{i+j}^0$ \\
    $dd_i'$ & & $6dd_{i+j}^0$ \\
    $dd_i''$ & & $6dd_{i+j}^0$ \\
    $ds_i'$ & & $6dd_{i+j}^0$ \\
    $sd_i'$ & & $6sd_{i+j}^0$ \\
    $ss_i^1$ & & $6sd_{i+j}^0$ \\
    $ss_i^2$ & & $6sd_{i+j}^0$ \\
    \bottomrule
  \end{tabular} $\quad$
  \begin{tabular}{ccl}
    \toprule
    $G(u,v)$ & $H(s,t)$ & $G(s,t)\ast H(u,v)$ \\
    \midrule
    $dd_i^0$ & \multirow{10}{*}{$ss^0$} & $6ds_{i+j}^0$ \\
    $ds_i^0$ & & $6ds_{i+j}^0$ \\
    $sd_i^0$ & & $6ss_{i+j}^0$ \\
    $ss_i^0$ & & $6ss_{i+j}^0$ \\
    $dd_i'$ & & $6ds_{i+j}^0$ \\
    $dd_i''$ & & $6ds_{i+j}^0$ \\
    $ds_i'$ & & $6ds_{i+j}^0$ \\
    $sd_i'$ & & $6ss_{i+j}^0$ \\
    $ss_i^1$ & & $6ss_{i+j}^0$ \\
    $ss_i^2$ & & $6ss_{i+j}^0$ \\
    \midrule
    $dd_i^0$ & \multirow{10}{*}{$dd'$} & $4dd_{i+j}^0 + 2dd_{i+j+2}^0 $ \\
    $ds_i^0$ & & $6dd_{i+j}^0$ \\
    $sd_i^0$ & & $4sd_{i+j}^0 + 2sd_{i+j+2}^0$ \\
    $ss_i^0$ & & $6sd_{i+j}^0$ \\
    $dd_i'$ & & $dd_{i+j}' + 3dd_{i+j}^0 + 2dd_{i+j+2}'$ \\
    $dd_i''$ & & $2dd_{i+j}^0 + 2dd_{i+j}' + 2sd_{i+j+2}'$ \\
    $ds_i'$ & & $3dd_{i+j}^0 + 3dd_{i+j}'$ \\
    $sd_i'$ & & $sd_{i+j}' + 3sd_{i+j}^0 + 2sd_{i+j+2}^0$ \\
    $ss_i^1$ & & $3sd_{i+j}^0 + 3sd_{i+j}'$ \\
    $ss_i^2$ & & $2sd_{i+j}^0 + 2sd_{i+j}' + 2dd_{i+j}^0$ \\
    \midrule
    $dd_i^0$ & \multirow{10}{*}{$dd''$} & $4dd_{i+j}^0 + 2ds_{i+j+2}^0 $ \\
    $ds_i^0$ & & $6dd_{i+j}^0$ \\
    $sd_i^0$ & & $4sd_{i+j}^0 + 2ss_{i+j+2}^0$ \\
    $ss_i^0$ & & $6sd_{i+j}^0$ \\
    $dd_i'$ & & $2dd_{i+j}^0 + 2dd_{i+j}' + 2ds_{i+j+2}'$ \\
    $dd_i''$ & & $4dd_{i+j}' + 2ss_{i+j+2}^2$ \\
    $ds_i'$ & & $6dd_{i+j}'$ \\
    $sd_i'$ & & $2sd_{i+j}' + 2sd_{i+j}^0 + 2ss_{i+j+2}^1$ \\
    $ss_i^1$ & & $6sd_{i+j}'$ \\
    $ss_i^2$ & & $4sd_{i+j}' + 2dd_{i+j}''$ \\
    \bottomrule
  \end{tabular}
  \caption{Partial embedding type transitions for $dd^0, ds^0, sd^0, ss^0, dd' ,dd''$.} \label{tab:mult1}
\end{table}

\begin{table}[!ht]
  \centering
  \addtolength{\tabcolsep}{-1pt}
  \begin{tabular}{ccl}
    \toprule
    $G(u,v)$ & $H(s,t)$ & $G(s,t) \ast H(u,v)$ \\
    \midrule
    $dd_i^0$ & \multirow{10}{*}{$ds'_j$} & $4ds_{i+j}^0 + 2ds_{i+j+2}^0$ \\
              $ds_i^0$ & & $6ds_{i+j}^0$ \\
              $sd_i^0$ & & $4ss_{i+j}^0 + 2ss_{i+j+2}^0$ \\
              $ss_i^0$ & & $6ss_{i+j}^0$ \\
              $dd_i'$ & & $ds_{i+j}' + 3ds_{i+j}^0 + 2ds_{i+j+2}'$ \\
              $dd_i''$ & & $2ds_{i+j}^0 + 2ds_{i+j}' + 2ss_{i+j+2}^1$ \\
              $ds_i'$ & & $3ds_{i+j}^0 + 3ds_{i+j}'$ \\
              $sd_i'$ & & $2ss_{i+j}' + 2ss_{i+j}^0 + 2ss_{i+j+2}^1$ \\
              $ss_i^1$ & & $6sd_{i+j}'$ \\
              $ss_i^2$ & & $3ss_{i+j}^0 + 3ss_{i+j}^1$ \\
    \midrule
              $dd_i^0$ & \multirow{10}{*}{$sd'_j$} & $6dd_{i+j}^0$ \\
              $ds_i^0$ & & $6dd_{i+j}^0$ \\
              $sd_i^0$ & & $6sd_{i+j}^0$ \\
              $ss_i^0$ & & $6sd_{i+j}^0$ \\
              $dd_i'$ & & $3dd_{i+j}^0 + 3dd_{i+j}'$ \\
              $dd_i''$ & & $6dd_{i+j}'$ \\
              $ds_i'$ & & $6sd_{i+j}'$ \\
              $sd_i'$ & & $6dd_{i+j}'$ \\
              $ss_i^1$ & & $6sd_{i+j}'$ \\
              $ss_i^2$ & & $6sd_{i+j}'$ \\
    \bottomrule
  \end{tabular}
  \begin{tabular}{ccl}
    \toprule
    $G(u, v)$ & $H(s, t)$ & $G(s, t) \ast H(u, v)$ \\
    \midrule
    $dd_i^0$ & \multirow{10}{*}{$ss^1_j$} & $6ds_{i+j}^0$ \\
    $ds_i^0$ & & $6ds_{i+j}^0$ \\
    $sd_i^0$ & & $6ss_{i+j}^0$ \\
    $ss_i^0$ & & $6ss_{i+j}^0$ \\
    $dd_i'$ & & $3ds_{i+j}^0 + 3ds_{i+j}'$ \\
    $dd_i''$ & & $6ds_{i+j}'$ \\
    $ds_i'$ & & $6ds_{i+j}'$ \\
    $sd_i'$ & & $3ss_{i+j}^0 + 3ss_{i+j}^1$ \\
    $ss_i^1$ & & $6ss_{i+j}^1$ \\
    $ss_i^2$ & & $6ss_{i+j}^1$ \\
    \midrule
    $dd_i^0$ & \multirow{10}{*}{$ss^2_j$} & $4ds_{i+j}^0 + 2dd_{i+j}^0 $ \\
    $ds_i^0$ & & $6ds_{i+j}^0$ \\
    $sd_i^0$ & & $4ss_{i+j}^0 + 2sd_{i+j}^0$ \\
    $ss_i^0$ & & $6ss_{i+j}^0$ \\
    $dd_i'$ & & $2ds_{i+j}' + 2ds_{i+j}^0 + 2dd_{i+j}'$ \\
    $dd_i''$ & & $4ds_{i+j}' + 2dd_{i+j}'$ \\
    $ds_i'$ & & $6ds_{i+j}'$ \\
    $sd_i'$ & & $2ss_{i+j}' + 2ss_{i+j}^0 + 2sd_{i+j}'$ \\
    $ss_i^1$ & & $6ss_{i+j}^1$ \\
    $ss_i^2$ & & $4ss_{i+j}^1 + 2ss_{i+j}^2$ \\
    \bottomrule
  \end{tabular}
  \addtolength{\tabcolsep}{1pt}
  \caption{Partial embedding type transitions for $ds', sd', ss^1, ss^2$.} \label{tab:mult2}
\end{table}

Let $dd_n^0(x), ds_n^0(x), sd_n^0(x), ss_n^0(x), dd_n'(x), dd_n''(x), ds_n'(x), sd_n'(x),  ss_n^1(x), ss_n^2(x)$ denote the ten partial genus polynomials of $G_n(u_1, v_n)$. Expressing \Cref{tab:mult1,tab:mult2} under the form of generating functions, we derive the following recursive equations for the partial genus polynomials of $G_n(u_1, v_n)$.
\begin{align*}
  dd_n^0(x) =
  &\ dd_{n-1}^0(x) \left[ (4 + 2x^2) dd^0(x) + 6sd^0(x) + (4 + 2x^2)dd^{\prime}(x) + 4dd''(x) + 6sd^{\prime}(x)\right]\\
  &+ ds_{n-1}^0(x)\left[6dd^0(x) + 6sd^0(x) + 6dd^{\prime}(x) + 6dd''(x) + 6sd^{\prime}(x)\right]\\
  &+ dd_{n-1}^{\prime}(x)\left[(4 + 2x^2)dd^0(x) + 6sd^0(x) + 3dd^{\prime}(x) + 2dd''(x) + 3sd^{\prime}(x)\right] \\
  &+ dd_{n-1}''(x)\left[(4 + 2x^2)dd^0(x) + 6sd^0(x) + 2dd^{\prime}(x)\right]\\
  &+ ds_{n-1}^{\prime}(x)\left[6dd^0(x) + 6sd^0(x) + 3dd^{\prime}(x)\right] \displaybreak[0] \\
  ds_n^0(x) = &\ dd_{n-1}^0(x) [(4 + 2x^2)ds^0(x) + 6ss^{0}(x) + 2x^2dd''(x) + (4 + 2x^2)ds^{\prime}\\&  + 6ss^1(x) + (4 + 2x)ss^2(x)]\\
  &+ ds_{n-1}^0(x)\left[6ds^0(x) + 6ss^0(x) + 6ds^{\prime}(x) + 6ss^{1}(x) + 6ss^{2}(x)\right]\\
  &+ dd_{n-1}^{\prime}(x)\left[(4 + 2x^2)ds^0(x) + 6ss^0(x) + 3ds^{\prime}(x) + 3ss^{1}(x) + 2ss^{2}(x)\right] \\
  &+ dd_{n-1}''(x)\left[(4 + 2x^2)ds^0(x) + 6ss^0(x) + 2ds^{\prime}(x)\right]\\
  &+ ds_{n-1}^{\prime}(x)\left[6ds^0(x) + 6ss^0(x) + 3ds^{\prime}(x)\right] \displaybreak[0] \\
  sd_n^0(x) = &\ sd_{n-1}^0(x)\left[(4 + 2x^2)dd^0(x) + 6sd^0(x) + (4 + 2x)dd^{\prime}(x) + 4dd''(x) + 6sd^{\prime}(x)\right]\\
  &+ ss_{n-1}^0(x)\left[6dd^0(x) + 6sd^0(x) + 6dd^{\prime}(x) + 6dd''(x) + 6sd^{\prime}(x)\right]\\
  &+ sd_{n-1}^{\prime}(x)\left[(4 + 2x^2)dd^0(x) + 6sd^0(x) + 2dd''(x)\right]\\
  &+ ss_{n-1}^{1}(x)\left[(6dd^0(x) + 6sd^0(x) + 3dd^{\prime}(x)\right]\\
  &+ ss_{n-1}^{2}(x)\left[(4 + 2x^2)dd^0(x) + 6sd^0(x) + 2dd^{\prime}(x)\right] \displaybreak[0] \\
  ss_{n}^0(x) = &\ sd_{n-1}^0(x)[(4 + 2x^2)ds^0(x) + 6ss^0(x) + 2dd''(x) + (4 + 2x^2)ds^{\prime}(x)\\& + 6ss^1(x) + (4 + 2x)ss^2(x)]\\
  &+ ss_{n-1}^0(x)\left[6ss^0(x) + 6ds^0(x) + 6ds^{\prime}(x) + 6ss^1(x) + 6ss^2(x) \right]\\
  &+ sd_{n-1}^{\prime}(x)\left[6ss^0(x) + (4 + 2x^2)ds^0(x) + 2ds^{\prime}(x) + 3ss^1(x) + 2ss^2(x)\right]\\
  &+ ss_{n-1}^{1}(x)\left[6ss^0(x) + 6ds^0(x) + 3ds^{\prime}(x)\right]\\
  &+ ss_{n-1}^{2}(x)\left[6ss^0(x) + (4 + 2x)ds^0(x) + 2ds^{\prime}(x)\right] \displaybreak[0] \\
  dd_n^{\prime}(x) = &dd_{n-1}^{\prime}(x)\left[(1 + 2x^2)dd^{\prime}(x) + 2dd''(x) + 3sd^{\prime}\right]\\
  &+ dd_{n-1}''(x)\left[2dd^{\prime}(x) + 4dd''(x) + 6sd^{\prime}\right]\\
  &+ ds_{n-1}^{\prime}(x)\left[3dd^{\prime}(x) + 6dd''(x) + 6sd^{\prime}\right] \displaybreak[0] \\
  dd_n''(x) = &\ \frac{2}{x}ss_{n-1}^2(x)ss^2(x) \displaybreak[0] \\
  ds_n^{\prime}(x) = &\ dd_{n-1}^{\prime}(x)\left[(1 + 2x^2)ds^{\prime}(x) + 3ss^1(x) + (2 + 2x)ss^2(x) + 2x^2dd''(x)\right]\\
  &+ dd_{n-1}''(x)\left[6ss^1(x) + 2ds^{\prime}(x) + 4ss^2(x)\right]\\
  &+ ds_{n-1}^{\prime}(x)\left[3ds^{\prime}(x) + 6ss^1(x) + 6ss^2(x)\right] \displaybreak[0] \\
  sd_n^{\prime}(x) = &\ dd_{n-1}''(x)\left[2x^2dd^{\prime}(x)\right]\\
  &+ sd_{n-1}^{\prime}(x)\left[(1 + 2x^2)dd^{\prime}(x) + 2dd''(x) + 6sd^{\prime}(x)\right]\\
  &+ ss_{n-1}^1(x)\left[3dd^{\prime}(x) + 6dd''(x) + 6sd^{\prime}(x)\right]\\
  &+ ss_{n-1}^2(x)\left[(2 + 2x)dd^{\prime}(x) + 4dd''(x) + 6sd^{\prime}(x)\right] \displaybreak[0]\\
  ss_n^1(x) = &\ dd_{n-1}''(x)\left[2x^2ds^{\prime}(x)\right]\\
  &+ sd_{n-1}^{\prime}\left[2x^2dd''(x) + (2 + 2x^2)ds^{\prime}(x) + 3ss^1(x) + (2 + 2x)ss^2(x)\right]\\
  &+ ss_{n-1}^1(x)\left[3ds^{\prime}(x) + 6ss^1(x) + 6ss^2(x)\right]\\
  &+ ss_{n-1}^2(x)\left[2ds^{\prime}(x) + 6ss^1(x) + 4ss^2(x)\right] \displaybreak[0] \\
  ss_n^2(x) = &\ dd_{n-1}''(x)\left[2x^2dd''(x) + 2xss^{\prime}(x)\right]\\
  &+ ss_{n-1}^2(x)\left[2xdd''(x) + 2xds^{\prime}(x)\right]
\end{align*}

Let
\[
  V_n = (dd_n^0(x), ds_n^0(x), sd_n^0(x), ss_n^0(x), dd_n'(x), dd_n''(x), ds_n'(x), sd_n'(x),  ss_n^1(x), ss_n^2(x)),
\]
then the recursive equations for the partial genus polynomials of $G_n(u_1, v_n)$ can be transformed into matrix form as
\[
  V_n(x) = V_{n-1} M(x)
\]
where $M(x)$ is a matrix of order 10. We refer to $M(x)$ as the \textit{production matrix} (or \textit{transfer matrix}) of $G_n(u_1,v_n)$.

If we take the graph $P_{2}^2(u,v)$ as $H(u, v)$, then the corresponding $G_n(u_1, v_n)$ is $P_n^2$. There are two embeddings of $P_2^2$, one on the sphere and the other on the projective plane. Thus, the pEd-vector of $P_{2}^2$ is
\[
    \begin{pmatrix}
      0 & 0 & 0 & 0 & 0 & 1 & 0 & 0 & 0 & x
    \end{pmatrix},
\]
while the production matrix $M(x)$ for $P_n^2$ is
\begin{equation}\label{matrix}
  M(x) =
  \begin{pmatrix}
    4 & 4x + 4 x^2 & 0 & 0 & 0 & 0 & 0 & 0 & 0 & 0 \\
    6 & 6 x & 0 & 0 & 0 & 0 & 0 & 0 & 0 & 0 \\
    0 & 0 & 4 & 4x + 4x^2 & 0 & 0 & 0 & 0 & 0 & 0 \\
    0 & 0 & 6 & 6x & 0 & 0 & 0 & 0 & 0 & 0 \\
    2 & 2 x & 0 & 0 & 2 & 0 & 2x + 4 x^2 & 0 & 0 & 0 \\
    0 & 0 & 0 & 0 & 4 & 0 & 4 x & 0 & 0 & 4 x^2 \\
    0 & 0 & 0 & 0 & 6 & 0 & 6 x & 0 & 0 & 0 \\
    0 & 0 & 2 & 2 x & 0 & 0 & 0 & 2 & 2x + 4 x^2 & 0 \\
    0 & 0 & 0 & 0 & 0 & 0 & 0 & 6 & 6 x & 0 \\
    0 & 0 & 0 & 0 & 0 & 2 & 0 & 4 & 4 x & 2 x \\
  \end{pmatrix}
\end{equation}

The characteristic polynomial of $M(x)$ is
\begin{align*}
  &\quad t^{10} + t^9 (-12 - 26 x) + t^8 (52 + 240 x + 160 x^2) + t^7 (-96 - 728 x - 768 x^2 + 816 x^3) \\
  &+ t^6 (64 + 768 x + 208 x^2 - 8640 x^3 - 8304 x^4) \\
  &+ t^5 (-128 x + 1920 x^2 + 21216 x^3 + 29376 x^4 - 16416 x^5) \\
  &+ t^4 (-512 x^2 - 9216 x^3 + 4992 x^4 + 165888 x^5 + 155520 x^6) \\
  &+ t^3 (-18432 x^4 - 179712 x^5 - 165888 x^6 + 359424 x^7) \\
  &+ t^2 (-239616 x^6 - 1327104 x^7 - 884736 x^8) + t(-1327104 x^8 - 3317760 x^9) - 2654208 x^{10}.
\end{align*}
According to \cite{CG18}, the Euler-genus polynomial $E_{n}(x)$ of $C_n^2$ follows the given recursive equation.
  
  \begin{equation}
    E_{n}(x) = \sum_{j=1}^{10} b_{j}(x) E_{n-j}(x), \label{eq:bx}
  \end{equation}
  where
  \begingroup
  \allowdisplaybreaks
  \begin{align*}
    b_{1}(x) &= 12 + 26 x, \\
    b_{2}(x) &= -52 - 240 x - 160 x^2,\\
    b_{3}(x) &= 96 + 728 x + 768 x^2 - 816 x^3 \\
    b_{4}(x) &= -64 - 768 x - 208 x^2 +8640 x^3 + 8304 x^4, \\
    b_{5}(x) &= 128 x - 1920 x^2 - 21216 x^3 - 29376 x^4 + 16416 x^5, \\
    b_{6}(x) &= 512 x^2 + 9216 x^3 - 4992 x^4 - 165888 x^5 - 155520 x^6, \\
    b_{7}(x) &= 18432 x^4 + 179712 x^5 + 165888 x^6 - 359424 x^7, \\
    b_{8}(x) &= 239616 x^6 + 1327104 x^7 + 884736 x^8, \\
    b_{9}(x) &= 1327104 x^8 + 3317760 x^9, \\
    b_{10}(x) &= 2654208 x^{10}.
  \end{align*}
  \endgroup

\subsubsection{Adding a pair of parallel edges}

We now relate the Euler-genus polynomial of $C_n^2$ to that of $P^2_{n - 1}$ by looking at the effect of adding a pair of parallel edges between both ends $u, v$ of $P^2_{n - 1}$.

We consider a general embedded graph $G$ with two degree-2 vertices $u, v$. Suppose that $u$ and $v$ are incident to the same face $F$, and fix a corner of $F$ at each of $u$ and $v$ where a new $uv$-edge is to be inserted. With respect to the local orientations at $u$ and $v$, the chosen corners of $F$ have either the same or opposite orientations. We say that $F$ is of \emph{type $0$} in the former case and \emph{type $1$} in the latter; see Figure~\ref{Type1}. Equivalently, an untwisted $uv$-edge splits a type-$0$ face, whereas a twisted $uv$-edge splits a type-$1$ face.

%%%%%%%%%%%%%%%%%%%%%%%%%%%%%%%%%%%%%%%%%%%%%
\begin{figure}[htp]
\centering
  \includegraphics[width=1.25in]{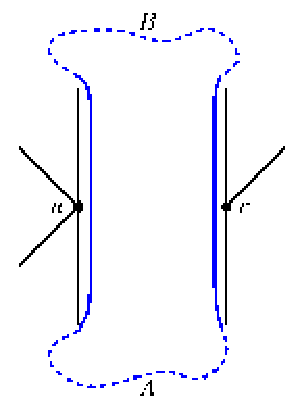} $\qquad \qquad$
  \includegraphics[width=1.5in]{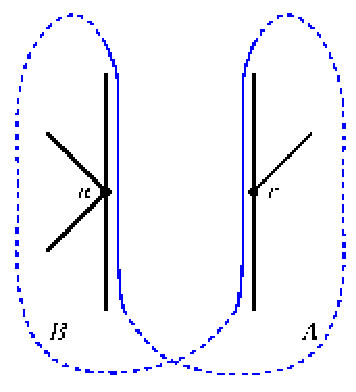}

\caption{Type $0$ (left) and Type $1$ (right).}
\label{Type1}
\end{figure}
%%%%%%%%%%%%%%%%%%%%%%%%%%%%%%%%%%%%%%%%%%%%%

The following edge-adding rules are fundamental principles in topological graph theory, some of which have been previously discussed in \cite{Rin77}.
\begin{description}
\item[R1] If $u, v$ lie on two different faces, then adding an edge $uv$ (resp. a twisted edge $uv^{\times}$) merges the two faces into a single one of type $0$ (resp. type $1$) in $G + uv$ (resp. $G + uv^{\times}$), thus $\eulergenus(G + uv)$ (resp. $\eulergenus(G + uv^{\times})$) is equal to $\eulergenus(G) + 1$.
\item[R2] If $u, v$ lie on the same face of type $0$ (resp. type $1$), then adding an edge $uv$ (a twisted edge $uv^{\times}$) splits the face into two in $G + uv$ (resp. $G + uv^{\times}$), thus $\eulergenus(G + uv^{\times})$ (resp. $\eulergenus(G + uv)$) is equal to $\eulergenus(G)$.
\item[R3] If $u, v$ lie on the same face of type $1$ (resp. type $0$), then adding an edge $uv$ (resp. a twisted edge $uv^{\times}$) increases its degree of by 2 in $G + uv$ (resp. $G + uv^{\times}$), thus $\eulergenus(G + uv)$ (resp. $\eulergenus(G + uv^{\times})$) is equal to $\eulergenus(G) + 1$.
\end{description}

Let $W(s, v)$ be a connected graph with two marked vertices $s, v$ of degree $2$. We add a pair of parallel edges to the vertices $s$ and $v$, resulting in a new graph $\overline{W}$. It is important to note that the degrees of both $s$ and $v$ are now 4 in $\overline{W}$. By following edge-adding rules and using the face-tracing algorithm, \Cref{tab:add-double-edge} shows how the Euler-genus of $\overline{W}$ varies from that of $W$ according to the partial Euler-genus type. In \Cref{tab:add-double-edge}, the sum of each row is 144, the total possibilities of inserting the pair of parallel edges. As an example, for the type $dd^0$, the Euler-genus of $\overline{W}$ increases from that of $W$ by $2$ for $32$ embeddings, by $3$ for another $32$ embeddings, and by $4$ for the remaining $80$ embeddings.

\begin{table}
  \centering
  \begin{tabular}{cccccc}
    \toprule
    \multicolumn{1}{c}{\multirow{2}{*}{$W(s,v)$}} & \multicolumn{5}{c}{The Euler-genus $\eulergenus(\overline{W})$ of $\overline{W}$}\\
    \cmidrule{2-6} & $\eulergenus(W)$ & $\eulergenus(W)+1$ & $\eulergenus(W)+2$ & $\eulergenus(W)+3$ & $\eulergenus(W)+4$ \\
    \midrule
    $dd^0$ &  &  & 32 & 32 & 80 \\
    $ds^0$ or $sd^0$ &  &  & 50  & 48 & 46 \\
    $ss^0$ &  &  & 72 & 72 &  \\
    $dd'$  & 2 & 6 & 52  & 44 & 40 \\
    $dd''$ & 6  &  16& 70 & 52 &  \\ 
    $ds'$ or $sd'$ & 6 & 18 & 72 & 48 &  \\ 
    $ss^1$  & 18 & 54 & 72 &  &  \\ 
    $ss^2$  & 20 & 56 & 68 &  &  \\
    \bottomrule
  \end{tabular}
\caption{The Euler-genus of $\overline{W}$ in terms of that of $W$ and the different types for $W(s, v)$.} \label{tab:add-double-edge}
\end{table}

From \Cref{tab:add-double-edge}, for $n \geq 2$, the Euler-genus polynomial $E_{n}(x)$ of $C^2_n$ equals
\[
\begin{pmatrix}
0 & 0 & 0 & 0 & 0 & 1 & 0 & 0 & 0 & x
\end{pmatrix}
M(x)^{n-1}\mathbf{c}(x),
\]
where $\mathbf{c}(x)=(c_1(x),\ldots,c_{10}(x))^{T}$ and
\begin{align*}
c_1(x)&=32x^2+32x^3+80x^4,
& c_6(x)&=6+16x+70x^2+52x^3,\\
c_2(x)&=50x^2+48x^3+46x^4,
& c_7(x)&=6+18x+72x^2+48x^3,\\
c_3(x)&=50x^2+48x^3+46x^4,
& c_8(x)&=6+18x+72x^2+48x^3,\\
c_4(x)&=72x^2+72x^2,
& c_9(x)&=18+54x+72x^2,\\
c_5(x)&=2+6x+52x^2+44x^3+40x^4,
& c_{10}(x)&=20+56x+68x^2.
\end{align*}
A direct computation gives the following initial values of $E_n(x)$.
\begin{equation}
  \left\{
  \begin{aligned} 
    E_1(x) &= 4 + 10x + 10x^2, \\[3pt]
    E_2(x) &= 6 + 36x + 126x^2 + 120x^3, \\[3pt]
    E_3(x) &= 8 + 84x + 576x^2 + 1444x^3 + 1344x^4, \\[3pt]
    E_4(x) &= 16 + 208x + 1944x^2 + 8128x^3 + 17960x^4 + 13216x^5, \\[3pt]
    E_5(x) &= 32 + 512x + 6304x^2 + 35792x^3 + 120224x^4 + 208272x^5 + 126528x^6, \\[3pt]
    E_6(x) &= 64 + 1216x + 20160x^2 + 145472x^3 + 634528x^4 + 1650112x^5 + 2334112x^6 \\
          &\quad{} + 1186304x^7, \\[3pt]
    E_7(x)&= 128 + 2816x + 64768x^2 + 573696x^3 + 3042048x^4 + 10201152x^5 \\
          &\quad{} + 21506560x^6 + 25230656x^7 + 11041792x^8, \\[3pt]
    E_8(x)&= 256 + 6400x + 213504x^2 + 2261504x^3 + 14003712x^4 + 56356352x^5 \\
          &\quad{} + 152367488x^6 + 266558464x^7 + 266050176x^8 + 102145536x^9, \\[3pt]
    E_9(x)&= 512 + 14336x + 730624x^2 + 9050112x^3 + 63676416x^4 + 294905856x^5 \\
          &\quad{} + 950924288x^6 + 2133587200x^7 + 3176284672x^8 + 2748807424x^9 \\
          &\quad{} + 941579264x^{10}, \\[3pt]
    E_{10}(x)&= 1024 + 31744x + 2601984x^2 + 36924416x^3 + 289603584x^4 + 1503739904x^5 \\
          &\quad{} + 5549844480x^6 + 14842849280x^7 + 28366170624x^8 + 36636175360x^9 \\
          &\quad{} + 27954014720x^{10} + 8652771328x^{11}.
  \end{aligned} \right. \label{eq:initial-values-En}
\end{equation}

\begin{thm} \label{thm:double-cycle-euler-genus}
  The Euler-genus distribution of $C_n^2$ is asymptotically normal with mean $(5/7) n$ and variance $(78/343) n$.
\end{thm}
\begin{proof}
  As in the proof of Theorem~5.1, we consider the generating function $E(t, x) = \sum_{n \ge 1} E_n(x) t^n$. It follows from the recurrence equation (\ref{eq:bx}) and the initial values (\ref{eq:initial-values-En}) that 
  \begin{align}
    E(t, x) &= \frac{A(t, x)}{B(t, x)} = \frac{A_1(t, x)}{B_1(t, x)} \label{eq:reduced}
  \end{align}
  is a rational function, where
  \begin{align*}
    \begin{split}
      A(t, x) &= \sum_{j=1}^{10} E_j(x) t^j - \sum_{j=1}^9 \sum_{k=1}^{10 - j} b_j(x) E_k(x) t^{j+k} \\
              &=(24x^2t^2 + 6xt + 2t - 1)(24x^2t^2 + 6xt + 4t - 1)  A_1(t, x),
    \end{split}
    \displaybreak[0] \\
    \begin{split}
      A_1(t, x) &= 1152 x^4 (6x^3 - 11x^2 - 2x - 1) t^6 + 32 x^2 (100x^4 - 276x^3 - 91x^2 - 36x - 9) t^5 \\
                &\quad + 8 (130x^5 - 153x^4 - 153x^3 - 49x^2 - 13x - 2) t^4 \\
                &\quad + 4 (74x^4 + 130x^3 + 27x^2 - 14x - 1) t^3 \\
                &\quad + 2 (10x^3 + 37x^2 + 40x + 9) t^2 - 2 (5x^2 + 5x + 2) t,
    \end{split}
    \displaybreak[0] \\
    \begin{split}
      B(t, x) &= 1 - \sum_{j=1}^{10} b_j(x) t^j \\
              &=(24x^2t^2 + 6xt + 2t - 1) (24x^2t^2 + 6xt + 4t - 1) B_1(t, x),
    \end{split}
    \displaybreak[0] \\
    \begin{split}
      B_1(t, x) &= (1 + 2xt) (1 - 4xt) (24x^2t^2 + 6xt + 2t - 1) (24x^2t^2 + 6xt + 4t - 1).
    \end{split}
  \end{align*}
  Thus, the singularities of $E(t, x)$ are the roots of $B_1(t, x)$, among which the dominant singularity is
  \[
    t = \rho(x): = \frac{\sqrt{33x^2 + 12x + 4} - 3x - 2}{24x^2},
  \]
  which is a solution of $24x^2t^2 + 6xt + 4t - 1 = 0$ for $t$. Using {\em Maple}, we check that all conditions of \Cref{thm:rational-asympt-normal} holds, and a direct computation thus gives our claim.
\end{proof}

As in \Cref{thm:double-cycle-genus}, the asymptotic normality in \Cref{thm:double-cycle-euler-genus} can be strengthen to a local limit law by applying \cite[Theorem~3]{Ben73}. However, the conditions are more complicated to check, with a much more complicated formula as result.

We note that \Cref{eq:reduced} implies a recurrence of smaller order than that of \Cref{eq:bx}:
\[
  E_{n}(x) = \sum_{j=1}^{6} \beta_{j}(x) E_{n-j}(x),
\]
where
\begin{align*}
  &\beta_{1}(x) = 14x + 6, &&\beta_{2}(x) = -4(x^2 + 12x + 2), \\
  &\beta_{3}(x) = -8x (51x^2 + 15x - 2), &&\beta_{4}(x) = -32x^2 (3x^2 - 18x - 2), \\
  &\beta_{5}(x) = 1152x^4 (3x + 1), &&\beta_{6}(x) = 4608 x^{6}.
\end{align*}
The recursion above implies that there exists a production matrix $M'$ of order $6$ for the double-edge cycle graph. It is thus interesting to see if there is a production matrix $M_*$ for $E_{n}(x)$ that is \tdef{primitive} after proper scaling, that is, for some $n \geq 1$,  all entries of $M_*^n$ are positive. It is equivalent to say that the directed graph associated to $M_*$ is strongly connected. However, this is not the case for the matrix $M(x)$ corresponding to the recurrence \Cref{eq:bx} for $E_{n}(x)$.

\begin{prop}\label{prop:not-prim}
  The stochastic matrix $M(x)$ of order $10$ associated with the Euler-genus distributions of double-edge cycle graphs is not primitive.
\end{prop}
\begin{proof}
  It is clear that the graph associated to $M(x)$ is not strongly connected, with the vertices corresponding to the first and the second lines forming a non-trivial strongly connected component.
\end{proof}

It should be noted that calculating the production matrix $M(x)$ for an $H(u, v)$-linear family of graphs (or the slightly more general $H$-linear family studied in the next section) is a challenging task. In \cite{GKMT18}, Gross, Khan, Tucker and Mansour presented a computational method for determining $M(x)$. It had been observed that in all known cases, the production matrix $M(1)$ is primitive. Stahl \cite{Sta91a} raised the question of whether this is always true. By \Cref{prop:not-prim}, the stochastic matrix of an $H$-linear family of graphs may not be primitive, which appears to be the first instance in this field. However, the stochastic matrix of order $6$ associated to the Euler-genus distributions of double-edge cycle graphs is primitive. Now we present the two following questions.

\begin{pro} \label{prob:stochastic-matrix-primitive}
  Is there always a stochastic matrix associated with an $H$-linear family of graphs that is primitive?
\end{pro}

If the answer is affirmative for this question, then we have the following natural question.

\begin{pro}
  Given an $H$-linear family of graphs, can we determine whether the associated minimum order stochastic matrix is primitive without computing it?
\end{pro}

\section{Genus generating functions of \texorpdfstring{$H$}{H}-linear and \texorpdfstring{$H$}{H}-circular graphs} \label{sec:h-families}

In this section, we generalize the computation performed in \Cref{sec:double-edge-cycle} on the more general case of \emph{$H$-linear} and \emph{$H$-circular} graphs. Under the framework of group algebra, we show that they always have rational generating functions for genus and Euler-genus. This framework also allows us to implement an algorithm that computes these generating functions automatically, to which we may apply \Cref{thm:rational-asympt-normal} to show the asymptotic normality of the associated embedding distributions while computing their mean and variance semi-automatically, as in \Cref{sec:double-edge-cycle}.

Given two graphs $H_1$ and $H_2$, with $U_1$ (resp. $U_2$) a non-empty subset of vertices in $H_1$ (resp. $H_2$), for any bijection $\varphi: U_1 \to U_2$, which is called a \tdef{gluing}, the \tdef{vertex-amalgamation} (or simply \tdef{amalgamation}) of $H_1$ and $H_2$ along $\varphi$, denoted by $H_1 \vertamal_\varphi H_2$ (or simply $H_1 \vertamal H_2$ when the gluing $\varphi$ is clear from the context), is obtained by identifying $u$ and $\varphi(u)$ for every $u \in U_1$ in the disjoint union of $H_1$ and $H_2$. It is clear that the amalgamation operation is associative. Similarly, we define the \tdef{self-amalgamation} of a graph $H$ along a gluing $\varphi: U_1 \to U_2$ with $U_1, U_2$ non-intersecting non-empty subsets of vertices in $H$.

Given a graph $H$ with a gluing $\varphi: U_1 \to U_2$ with $U_1, U_2$ non-intersecting non-empty subsets of vertices in $H$, we define an \tdef{$H$-linear family} $\graphseq(H, \varphi) = (G_n)_{n \geq 1}$ inductively by taking $G_1 = H$ and $G_{n+1} = G_n \vertamal H$ along the gluing that sends the vertices originally in $U_1$ that are not yet identified under amalgamation to the vertices originally in $U_2$ in the new copy of $H$. See \Cref{fig:h-linear-circular} for an example of an $H$-linear family. The construction of H-linear families goes back to Stahl \cite{Sta91a} and was later extended by Gross, Khan, Mansour, and Tucker \cite{GKMT18}. As another example, the double-edge path graphs $(P_n^2)_{n \geq 1}$ is a $P_1^2$-linear family, glued along the two vertices in $P_1^2$. Given an $H$-linear family $\graphseq = (G_n)_{n \geq 1}$ constructed over some gluing $\varphi: U_1 \to U_2$, we define a \tdef{$H$-circular family} $\graphseq^\circ(H, \varphi) = (G_n^\circ)_{n \geq 1}$ by taking $G_n^\circ$ to be the self-amalgamation of $G_n$ along the gluing that sends the vertices originally in $U_1$ that are not yet identified under amalgamation (thus from the last added copy of $H$) to those originally in $U_2$ again not yet identified (thus from the first copy of $H$). See \Cref{fig:h-linear-circular} for an example of an $H$-circular family. For another example, the family of double-edge cycle graphs $(C_n^2)_{n \geq 1}$ is a $P_1^2$-circular family, again glued along the two vertices in $P_1^2$.

\begin{figure}
  \centering
  \includegraphics[width=0.8\textwidth]{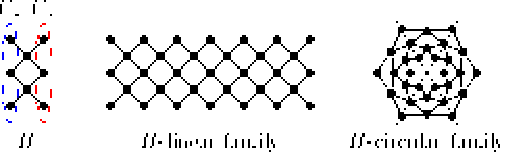}
  \caption{Example of $H$-linear and $H$-circular graphs.}
  \label{fig:h-linear-circular}
\end{figure}

In the original article \cite{Sta91a} defining $H$-linear families, several examples were also given. Such families capture the notion of uni-directional periodicity in graphs. The notion of $H$-circular family is a natural ring-like extension to it (see, \emph{e.g.}, \cite{CG19}). Other notable examples of $H$-linear families include the half-ladders studied in \Cref{sec:ring-like} and grid graphs with a fixed side, while bar-rings composed by identical graphs and grids graphs with a fixed side on a torus are natural examples of $H$-circular families.

Given a vertex $u$ in a graph $H$, we define the \tdef{blow-up} of $H$ at $u$, denoted by $H_{*u}$, to be the graph obtained by removing $u$ and adding a distinct endpoint for each of its edges. Note that the blow-up of a graph is not necessarily connected. By abuse of notation, we extend the definition of blow-ups to a subset $U$ of vertices in $H$ by performing the same operation on every vertex in $U$ simultaneously, and we denote by $H_{*U}$ the graph obtained. For a further abuse of notation, given a graph $H$ and a self-gluing $\varphi: U_1 \to U_2$, we say that the blow-up of $H$ along $\varphi$, denoted by $H_{*\varphi}$, to be the blow-up $H_{*(U_1 \cup U_2)}$. %In this case, we denote by $D_\varphi$ the darts associated to the newly added vertices in $H_{*\varphi}$.

We now consider rotation systems of a graph and its blow-up. Recall that $\rotsys(G)$ is the set of rotation systems of a graph $G$.

\begin{lemma} \label{lem:blow-up-genus}
  Let $H$ be a connected graph, $U$ a subset of vertices of $H$, and $D_U$ the darts of new vertices in the blow-up $H_{*U}$. Let $\mathcal{C}_U$ be the set of permutations in which each cycle is formed by exactly all the darts of a certain vertex $u \in U$. Then there is a bijection between $\rotsys(H)$ and $\mathcal{C}_U \times \rotsys(H_{*U})$, which sends $(\sigma, \tau) \in \rotsys(H)$ to $\pi, (\sigma', \tau) \in \mathcal{C}_U \times \rotsys(H_{*U})$ such that $\sigma = \pi \sigma'$.
\end{lemma}
\begin{proof}
  Let $(\sigma, \tau) \in \rotsys(H)$. We can write $\sigma = \pi \sigma'$ such that $\pi$ contains the cycles corresponding to vertices in $U$. It is clear that $\pi \in \mathcal{C}_U$. Then, $(\sigma', \tau)$ is a rotation system of $H_{*U}$, as the removal of the cycles in $\pi$ from $\sigma$ is equivalent to making all darts related to vertices in $U$ a fixed point in $\sigma'$. Conversely, given $(\sigma', \tau) \in \rotsys(H)$ and $\pi \in \mathcal{C}_U$, by the definition of $H_{*U}$, all darts in $D_U$ are fixed points in $\sigma'$, thus $(\pi, \sigma, \tau)$ is a rotation system of $H_{*U}$, with the cycles in $\pi$ giving the cyclic order of darts around each $u \in U$.
\end{proof}

To use \Cref{lem:blow-up-genus} to compute embedding distributions of $H$-linear and $H$-circular families, we need some algebraic setup. Given a finite set $E$, we denote by $S_E$ the symmetric group over elements of $E$. We note that the group structure of $S_E$ depends only on the size of $E$. A \tdef{group ring} of $S_E$ over the ring $R$, denoted by $R[S_E]$, is the $R$-module with all elements in $S_E$ as a formal basis, endowed with the multiplication given by the linear extension of the group law of $S_E$. In the following, we will mostly be working with the ring $R = \rational[x, x^{-1}]$ of Laurent polynomials in $x$. Given $E' \subseteq E$ and $\sigma \in S_E$, the \tdef{projection} of $\sigma$ to $S_{E'}$, denoted by $\proj_{E'}^{E}(\sigma)$, is the permutation $\sigma'$ obtained from the restriction of the cyclic notation of $\sigma$ to elements in $E'$. For instance, for $E' = \{1, 2, 3\}$ and $E = \{1, 2, 3, 4, 5\}$, for $\sigma = (1, 2, 5)(3, 4)$, we have $\proj_{E'}^E(\sigma) = (1, 2)(3)$. With abuse of notation, we also consider $\proj_{E'}^E$ as a function from $R[S_E]$ to $R[S_{E'}]$ by extending it linearly. For a graph $G$, a \tdef{partial rotation system} is the projection of a rotation system on $G$ to a subset of darts. Such systems were first proposed by Gross, Mansour, Tucker and Wang in \cite{gross-mansour-tucker-wang-local-log-concavity} to study a local variant of the conjecture on log-concavity of genus distribution in \cite{GroRobTuc89}, and this variant was affirmed by Féray in \cite{feray-local-log-concave} by connecting it to a result of Stanley \cite{stanley-cycle-product}.

For a graph $H$ with $D$ the set of its darts, we define $\faceelem(H) \in \rational[x, x^{-1}] [S_D]$ to be the formal sum of the permutation representing faces in every rotation system of $H$. Formally,
\[
  \faceelem(H) = \sum_{(\sigma, \tau) \in \rotsys(H)} \sigma \tau.
\]
We note that $\faceelem(H)$ is well-defined even when $H$ is not connected. Given $\faceelem(H)$, we can compute the genus polynomial of $H$. For $\sigma \in S_E$, we define $\cyc(\sigma)$ to be the number of cycles in $\sigma$. Using this formalism, we can restate \Cref{lem:blow-up-genus} as follows.

\begin{coro} \label{coro:blow-up-genus}
  For $H$ a connected graph, $U$ a subset of vertices in $H$, and $D_U$ the set of new darts in the blow-up $H_{*U}$, we have
  \[
    \faceelem(H) = \mathbf{C}_U \faceelem(H_{*U}),
  \]
  where $\mathbf{C}_U$ is the formal sum of permutations of the form $\prod_{u \in U} \sigma_u$ where each $\sigma_u$ is a cyclic permutation of all darts containing the vertex $u$.
\end{coro}
\begin{proof}
  This is a consequence of \Cref{lem:blow-up-genus} and the definition of $\faceelem$.
\end{proof}

We have the following result to compute $\faceelem(H_1 \vertamal_\varphi H_2)$.

\begin{prop} \label{prop:amalgamation-faces}
  Let $H^{(1)}, H^{(2)}$ be two connected graphs, and $\varphi: U_1 \to U_2$ a gluing. We denote by $D_1, D_2$ the set of all darts in $H^{(1)}, H^{(2)}$ respectively. Let $H = H^{(1)} \vertamal_\varphi H^{(2)}$, and $U$ the set of vertices in $H$ from identification of vertices in $U_1$ and $U_2$. We take $D$ to be the disjoint union of $D_1$ and $D_2$ and identify elements in $D$ with darts in $H$. We have
  \[
    \faceelem(H) = \mathbf{C}_U \faceelem(H^{(1)}_{*U_1}) \faceelem(H^{(2)}_{*U_2}),
  \]
  where $\mathbf{C}_U$ is the formal sum of all permutations $\prod_{u \in U} \sigma_u$ where each $\sigma_u$ is a cyclic permutation of all darts of the vertex $u$ in $H$. Here, the multiplication should be considered to take place in $\rational[x, x^{-1}] [S_D]$.
\end{prop}
\begin{proof}
  By \Cref{coro:blow-up-genus}, we have $\faceelem(H) = \mathbf{C}_U \faceelem(H_{*U}) = \mathbf{C}_U \faceelem(H^{(1)}_{*U_1}) \faceelem(H^{(2)}_{*U_2})$. The second equality comes from the fact that $H_{*U}$ is the disjoint union of $H^{(1)}_{*U_1}$ and $H^{(2)}_{*U_2}$.
\end{proof}

To compute the genus polynomial $\Gamma_{H}(x)$, given a set $E$ and a subset $E' \subseteq E$ we define the \tdef{face projection} $\faceproj_{E'}^E: \rational [x, x^{-1}] [S_E] \to \rational [x, x^{-1}] [S_{E'}]$, which is a linear map, by taking $\faceproj_{E'}^E(\sigma) = x^{\cyc_{E \setminus E'}(\sigma)} \proj_{E'}^E(\sigma)$, where $\cyc_{E \setminus E'}(\sigma)$ is the number of cycles of $\sigma$ that do not contain any element in $E'$, and extending it linearly. We note that, for $E'' \subseteq E' \subseteq E$, we have $\faceproj_{E''}^{E} = \faceproj_{E''}^{E'} \circ \faceproj_{E'}^{E}$. The following indicates that we only need to consider $\faceelem(H)$ to compute $\Gamma_H(x)$.

\begin{prop} \label{prop:genus-from-face-element}
  Let $H$ be a connected graph with $v_H$ vertices and $e_H$ edges, and $E$ the set of all its darts. We have $\Gamma_{H}(x) = x^{1 + (e_H - v_H) / 2} (\faceproj_{\varnothing}^E \faceelem(H))(x^{-1/2})$.
\end{prop}
\begin{proof}
  This comes directly from Euler's formula and the definitions of $\faceelem$ and $\faceproj_{E'}^E$.
\end{proof}

To simplify computation, in the following we often work with projection of rotation systems to a smaller set of darts.

\begin{prop} \label{prop:prod-restriction}
  For a non-empty set $E$ and a non-empty subset $E' \subseteq E$, let $\sigma \in \rational [x, x^{-1}] [S_{E'}]$. Then for any $\pi \in \rational [x, x^{-1}] [S_E]$, we have
  \[
    \faceproj_{E'}^E(\sigma \pi) = \sigma \faceproj_{E'}^E(\pi).
  \]
\end{prop}
\begin{proof}
  By linearity, we only need to check the case where $\sigma \in S_{E'}$ and $\pi \in S_E$. We first consider the action of $\sigma$ on cycles of $\pi$. Let $\tau = (a, b)$ be a transposition in $S_{E'}$. If $a, b$ are in the same cycle $c$ in $\pi$, then in $\tau \pi$ the cycle $c$ is replaced by two cycles $c_a$ and $c_b$, with $c_a$ (resp. $c_b$) containing $a$ (resp. $b$). Otherwise, if $a, b$ are in different cycles $c_a, c_b$ respectively in $\pi$, then in $\tau \pi$ the two cycles $c_a, c_b$ are replaced by a cycle $c$ that contains both $a, b$. This observation shows that $\cyc_{E \setminus E'}(\tau \pi) = \cyc_{E \setminus E'}(\pi)$, as the cycles involved must contain some dart in $E'$. By decomposing $\sigma$ into transpositions, we have $\cyc_{E \setminus E'}(\sigma \pi) = \cyc_{E \setminus E'}(\pi)$. As $\sigma \in S_{E'}$, it is also clear that the multiplication by $\sigma$ commutes with $\proj_{E'}^E$. We then have
  \[
    \faceproj_{E'}^E(\sigma \pi) = x^{\cyc_{E \setminus E'}(\sigma \pi)} \proj_{E'}^E(\sigma \pi) = x^{\cyc_{E \setminus E'}(\pi)} \sigma \proj_{E'}^E(\pi) = \sigma \faceproj_{E'}^E(\pi). \qedhere
  \]
\end{proof}

\begin{thm} \label{thm:H-linear-rational}
  For an $H$-linear family $\graphseq = \graphseq(H, \varphi) = (G_n)_{n \geq 1}$ with the gluing $\varphi: U_1 \to U_2$, we define its generating function $F_{\graphseq}(t, x) = \sum_{n \geq 1} t^n \Gamma_{G_n}(x)$. Then, $F_{\graphseq}(t, x)$ is a rational fraction in $t, x$ with integer coefficients.  The same holds for the generating function $F_{\graphseq^\circ}(t, x) = \sum_{n \geq 1} t^n \Gamma_{G^\circ_n}(x)$ of an $H$-circular family $\graphseq^\circ = \graphseq^{\circ}(H, \varphi) = (G^\circ_n)_{n \geq 1}$.
\end{thm}
\begin{proof}
  Let $V_n$ be the set of vertices of $G_n$ originally in $U_1 \cup U_2$ that are not yet identified, and $G'_n = (G_n)_{*V_n}$ (resp. $G''_n = (G_n)_{*U_2}$) be the blow-up of $G_n$ at $V_n$ (resp. nodes originally from $U_2$ but not yet identified). By the definition of $H$-linear families, we have $G'_n = G''_{n-1} \vertamal_\varphi H_{*H_1}$. We denote by $U'_n$ the set of vertices identified in the amalgamation process to obtain $G'_n$. By \Cref{prop:amalgamation-faces}, we have $\faceelem(G'_n) = \mathbf{C}_{U'_n} \faceelem(G'_{n-1}) \faceelem(H_{*\varphi}) = \mathbf{C}_{U'_n}  \faceelem(H_{*\varphi}) \faceelem(G'_{n-1})$, the latter equality coming from the fact that $\faceelem(G'_{n-1})$ and $\faceelem(H_{*\varphi})$ are on disjoint supports.

  We denote by $D(G'_n)$ the set of darts of $G'_n$, by $D_n$ (resp. $D'_n$) the set of darts of $V_n$ (resp. $U'_n$) in $G'_n$. We also denote by $D(H)$ the set of darts of $H$, and by $D^H_n$ the set of darts in the copy of $H$ to construct $G'_n$ that are originally related to vertices of $H$ that are in $U_1 \cup U_2$.

  We now show that $\mathbf{A}_n = \faceproj^{D(G'_n)}_{D_{n}}(\faceelem(G'_n)) \in \rational [x, x^{-1}] [S_{D_{n}}]$ satisfies a linear recurrence. For $n \geq 2$, we have
  \begin{align*}
    \mathbf{A}_n &= \faceproj^{D(G'_n)}_{D_{n}}(\faceelem(G'_n)) \\
                 &= \faceproj^{D_{n} \cup D'_n}_{D_{n}} (\faceproj^{D(G'_n)}_{D_{n} \cup D'_{n}} (\mathbf{C}_{U'_n} \faceelem(H_{*\varphi}) \faceelem(G'_{n-1}))) \\
                 &= \faceproj^{D_{n} \cup D'_n}_{D_{n}} \left(\mathbf{C}_{U'_n} \cdot \faceproj^{D(G'_{n-1}) \cup D(H)}_{D_{n-1} \cup D^H_n} (\faceelem(H_{*\varphi}) \faceelem(G'_{n-1}))\right) \\
                 &= \faceproj^{D_{n} \cup D'_{n}}_{D_{n}} \left(\mathbf{C}_{U'_n} \cdot \faceproj^{D(H)}_{D^H_{n}} (\faceelem(H_{*\varphi})) \cdot \faceproj^{D(G'_{n-1})}_{D_{n-1}}( \faceelem(G'_{n-1})) \right) \\
                 &= \faceproj^{D_{n} \cup D'_{n}}_{D_{n}} \left(\mathbf{C}_{U'_n} \cdot \faceproj^{D(H)}_{D^H_{n}} (\faceelem(H_{*\varphi})) \cdot \mathbf{A}_{n-1} \right)
  \end{align*}
  The second equality comes from the property of $\faceproj_{E'}^E$ that $\faceproj_{E''}^E = \faceproj_{E''}^{E'} \circ \faceproj_{E'}^E$ for $E'' \subseteq E' \subseteq E$, the third from \Cref{prop:prod-restriction} and from the fact that $D_n \cup D'_n = D_{n-1} \cup D^H_n$, with disjoint unions on both sides, the fourth from the fact that $\faceelem(H_{*\varphi})$ and $\faceelem(G'_{n-1})$ are in group rings of symmetric groups with disjoint support. This is clearly a linear recurrence, as all the operators involved are linear, and although technically $\mathbf{A}_n$ and $\mathbf{A}_{n-1}$ have $S_{D_n}$ and $S_{D_{n-1}}$ as the underlying group respectively, they are in clear bijection induced by the bijection between $D_n$ and $D_{n-1}$ identifying vertices with the same origin in $H$. For simplicity, we say that all $\mathbf{A}_n$ are in $\rational [x, x^{-1}] [S_{D_*}]$ for the same $D_*$ without invoking the bijections identifying $D_n$ and $D_{n-1}$. We can do the same for $D'_n$ as the identified vertices are related by the same gluing $\varphi$, thus replacing every $D'_n$ by $D'_*$ without invoking the bijection.

  We now define the linear operator
  \begin{align*}
    \mathbf{T}_H: \rational [x, x^{-1}] [S_{D_*}] &\to \rational [x, x^{-1}] [S_{D_*}], \\
    \sigma &\mapsto \faceproj_{D_*}^{D_* \cup D'_*} (\mathbf{C}_{U'_n} \cdot \faceproj^{D(H)}_{D^H_{n}} (\faceelem(H_{*\varphi})) \cdot \sigma).
  \end{align*}
  The computation above thus shows that $\mathbf{A}_n = \mathbf{T}_H \mathbf{A}_{n-1}$. Let $\mathbf{F}(t) = \sum_{n \geq 1} t^n \mathbf{A}_n$, which is in $\rational [x, x^{-1}] [[t]] [S_{D_*}]$. As $G'_1 = H_{*U_2}$, by \Cref{coro:blow-up-genus} we have $\mathbf{A}_1 = \faceproj_{D_*}^{D(H)}(\mathbf{C}_{U_1} \faceelem(H_{*\varphi}))$. Hence,
  \[
    (1 - t \mathbf{T}_H) \mathbf{F}(t) = t \mathbf{A}_1.
  \]
  Hence, $\mathbf{F}(t)$ is governed by a linear system, meaning that all its coefficients $[\sigma]\mathbf{F}(t)$ are rational fractions in $t, x$ with a common denominator that divides $\det(1 - t \mathbf{T}_H)$.

  To obtain the generating function $F_\graphseq(t, x)$, we observe that by \Cref{coro:blow-up-genus}, we have $\faceelem(G_n) = \mathbf{C}_{U_1 \cup U_2} \mathbf{A}_n$, and we use \Cref{prop:genus-from-face-element} by observing that all operators involved are linear. We thus conclude that $F_\graphseq(t, x)$ is rational in $t, x^{1/2}$ with the denominator dividing $\det(1 - t \mathbf{T}_H)(x^{-1/2})$. However, as we see combinatorially that $F_\graphseq(t, x)$ is a series in $t$ with coefficients polynomial in $x$ and is analytic near $(t, x) = (0, 0)$, which implies that both the denominator and the numerator of $F_\graphseq(t, x)$ are polynomial in $x$.

  For $F_{\graphseq^\circ}$, we observe that $G^\circ_n = G_{n-1} \vertamal_{\varphi'} H$, with the gluing $\varphi'$ a ``double version'' of $\varphi$, sending vertices originally from $U_1$ not yet merged in $G_{n-1}$ (resp. $H$) to those from $U_2$ not yet merged in $H$ (resp. $G_{n-1}$). By \Cref{prop:amalgamation-faces}, with a computation similar to that for the recurrence on $\mathbf{A}_n$, we see that there is a linear operator $\mathbf{T}'_H$ such that $\Gamma_{G^\circ_n}(x) = \mathbf{T}'_H \mathbf{A}_{n-1}$. With the same arguments as in the case of $F_\graphseq(t, x)$, we know that $F_{\graphseq^\circ}(t, x)$ is also rational with the denominator a factor of $\det(1 - t \mathbf{T}_H)$.
\end{proof}

By replacing the symmetric group $S_D$ on the set of darts $D$ by the hyperoctahedral group $\mathcal{H}_D$ formed by sign-symmetric permutations over $\pm D = D \cup \{ \overline{x} \mid x \in D \}$, with adjustments on face counting and the element $\mathbf{C}_U$ in \Cref{coro:blow-up-genus,prop:amalgamation-faces}, we have the same result for Euler genus.

\begin{thm} \label{thm:H-linear-rational-euler}
  For an $H$-linear family $\graphseq = \graphseq(H, \varphi) = (G_n)_{n \geq 1}$ with the gluing $\varphi: U_1 \to U_2$, we define its generating function $F_{\graphseq}(t, x) = \sum_{n \geq 1} t^n \Eulergenus_{G_n}(x)$. Then, $F_{\graphseq}(t, x)$ is a rational fraction in $t, x$ with integer coefficients. The same holds for the generating function $F_{\graphseq^\circ}(t, x) = \sum_{n \geq 1} t^n \Eulergenus_{G^\circ_n}(x)$ of an $H$-circular family $\graphseq^\circ = \graphseq^{\circ}(H, \varphi) = (G^\circ_n)_{n \geq 1}$.
\end{thm}

The crucial point for the case of Euler-genus is the face projection $\faceproj^E_{E'}$. Instead of simply counting the number of cycles, it now needs to distinguish two types of cycles: those identical to their formal opposite, and those not. The former one should count as one face, and the latter only half, as its formal opposite is also a valid cycle describing the same face.

We note that \Cref{thm:H-linear-rational} and \Cref{thm:H-linear-rational-euler} is essentially a reformulation and a generalization of the production matrix approach in \cite{Sta91a,GKMT18,CG19}, but expressed using the formalism of partial rotation system and group algebra, which opens up a way to do the computations automatically without the manual and tedious ``face tracing'' to obtain the production matrix. We also note that both theorems are effective, and the parameter of complexity is the number of darts involved in the amalgamation. Hence, as long as there are relatively few such darts, the graph $H$ can be arbitrarily complex, in the limit of being able to compute $\faceelem(H_{*\varphi})$. An implementation for the case of genus distribution in Sagemath is available at \cite{genus-poly-h-family}.

Once the rational genus or Euler-genus generating function of an $H$-linear or an $H$-circular family is computed using \Cref{thm:H-linear-rational} or \Cref{thm:H-linear-rational-euler}, we may apply \Cref{thm:rational-asympt-normal} as in \Cref{sec:double-edge-cycle} to semi-automatically prove that the corresponding embedding distribution is asymptotically normal, while computing also its mean and variance. We may even try to apply \cite[Theorem~3]{Ben73} for a local limit law, but the expression obtained would be much more complicated, as we can see later that the mean and the variance can already be quite complicated for some innocent-looking examples.

As an example of application, using \Cref{thm:H-linear-rational}, we computed the genus distribution of $(C^3_n)_{n \geq 1}$, where $C^3_n$ is the \tdef{triple-edge cycle graph}, obtained by replacing each edge in the cycle graph $C_n$ by a triple of edges. It is the $H$-circular family with $H$ the graph with two vertices linked by 3 edges. Let $F_{C^3}(t, x) = \sum_{n \geq 1} \Gamma_{C^3_n}(x) t^n$, we have
\begingroup
\allowdisplaybreaks
\begin{align*}
  F_{C^3}(t, x) &= \frac{A(t, x)}{B(t, x)} \\
  A(t, x) &= 8t(4353564672000000 x^{13} t^{12} + 6530347008000000 x^{12} t^{12} + 1741425868800000 x^{12} t^{11} \\
                &\quad + 2176782336000000 x^{11} t^{12} + 4389844377600000 x^{11} t^{11} + 104001822720000 x^{11} t^{10}\\
                &\quad+ 798153523200000 x^{10} t^{11} + 1004947845120000 x^{10} t^{10} - 72559411200000 x^9 t^{11} \\
                &\quad - 45934138368000 x^{10} t^9 + 13907220480000 x^9 t^{10} + 77840123904000 x^9 t^9 \\
                &\quad - 33105231360000 x^8 t^{10} - 8270429184000 x^9 t^8 - 26086116096000 x^8 t^9 \\
                &\quad + 453496320000 x^7 t^{10} - 2348775014400 x^8 t^8 - 4237671168000 x^7 t^9 \\
                &\quad - 286738444800 x^8 t^7 - 2839894732800 x^7 t^8 + 292253184000 x^6 t^9 \\
                &\quad - 482184161280 x^7 t^7 + 69872025600 x^6 t^8 + 29358288000 x^7 t^6 + 77853000960 x^6 t^7 \\
                &\quad + 48876825600 x^5 t^8 + 3310009920 x^6 t^6 + 37246884480 x^5 t^7 - 503884800 x^4 t^8 \\
                &\quad + 1454500800 x^6 t^5 + 18235730880 x^5 t^6 + 1266710400 x^4 t^7 + 1357250256 x^5 t^5 \\
                &\quad - 302575824 x^4 t^6 - 151165440 x^3 t^7 - 67003200 x^5 t^4 - 84107160 x^4 t^5 \\
                &\quad - 186099120 x^3 t^6 - 25857144 x^4 t^4 - 161394768 x^3 t^5 - 3779136 x^2 t^6 - 1441584 x^4 t^3 \\
                &\quad - 39956652 x^3 t^4 + 2420280 x^2 t^5 - 1868292 x^3 t^3 + 2545344 x^2 t^4 - 23328 x t^5 \\
                &\quad + 102312 x^3 t^2 + 732510 x^2 t^3 + 159732 x t^4 + 87948 x^2 t^2 + 81432 x t^3 - 1740 x^2 t \\
                &\quad + 1233 x t^2 + 1134 t^3 - 1215 x t + 207 t^2 + 10 x - 105 t + 5) \\
  B(t, x) &= (6 x t + 1) (1 - 12 x t) (43200 x^3 t^3 + 2880 x^2 t^2 - 1080 x t^2 - 120 x t - 18 t + 1) \\
                &\quad \cdot (129600 x^4 t^4 + 21600 x^3 t^3 + 180 x^2 t^2 - 144 x t^2 - 60 x t + 1) \\
                &\quad \cdot (259200 x^4 t^4 + 60480 x^3 t^3 - 2160 x^2 t^3 + 2160 x^2 t^2 - 612 x t^2 - 114 x t - 6 t + 1)
\end{align*}
\endgroup
This result is obtained by first using the recurrence in \Cref{thm:H-linear-rational} to compute the first terms of the series $F_{C^3}(t, x)$, then use Padé's approximation to get the rational fraction expression, guaranteed by \Cref{thm:H-linear-rational}. The whole computation took less than one minute on an ordinary laptop. Applying \Cref{thm:rational-asympt-normal} to $F_{C^3}(t, x)$, we have the following asymptotic normality result.

\begin{thm} \label{thm:C3-normality}
  Let $X_n$ be the genus of a random embedding of $C^3_n$, then $X_n$ is asymptotically normal, with mean $\mathfrak{m}n + O(1)$ and variance $\mathfrak{v}n + O(1)$, where
  \[
    \mathfrak{m} = \frac{8}{11}, \quad \mathfrak{v} = \frac{150}{1331}.
  \]
\end{thm}

For larger examples, the genus generating function of the $H$-circular family in \Cref{fig:h-linear-circular} was computed on one core of a local compute server in about 3 days, and the highest power of $t$ (resp. $x$) of its denominator is $138$ (resp. $198$), while that for the quadruple-edge cycle graph $C^4_n$ was computed in about 7 days, with the denominator
\begingroup
\allowdisplaybreaks
\begin{align*}
  &(1 - 2304 x^3 t^2) (16128 x^3 t^2 - 96 x t - 1) (16128 x^3 t^2 - 72 x t - 1) (20736 x^3 t^2 - 1) \\
  &\cdot (2341011456 x^6 t^4 + 13934592 x^4 t^3 - (548352 x - 20736) x^2 t^2 - 288 x t + 1) \\
  &\cdot (245806202880 x^6 t^4 - (1430618112 x - 162570240) x^3 t^3 \\
  &\qquad - (7257600 x^2 - 80640 x + 80640) x t^2 - (3456 x + 96) t + 1) \\
  &\cdot (3964362440048640 x^9 t^6 - (48636854009856 x - 655483207680) x^6 t^5 \\
  &\qquad - (362856775680 x^2 - 130576416768 x + 10404495360) x^4 t^4 \\
  &\qquad + (2018193408 x^2 - 83220480 x + 1935360) x^2 t^3 \\
  &\qquad + (7273728 x^2 + 235008 x + 32256) x t^2 + (2328 x + 24) t - 1) \\ 
  &\cdot (5480334637123239936 x^{12} t^8 - 38057879424466944 x^{10} t^7 + \\
  &\qquad - (2871840485670912 x^2 - 27717575639040 x + 3745618329600) x^7 t^6 \\
  &\qquad + (20363678318592 x + 49941577728) x^6 t^5 \\
  &\qquad + (202558537728 x^2 - 37267734528 x + 2529460224) x^4 t^4 \\
  &\qquad - (793718784 x - 18137088) x^3 t^3 - (3299328 x^2 + 160704 x + 5184) x t^2 \\
  &\qquad - 912 x t + 1).
\end{align*}
\endgroup
The numerator is too long and omitted here. It is available in \cite{genus-poly-h-family}. We note that most of the computation time is spent on computing the needed first terms by iterating the operator $\mathbf{T}_H$. Again, by applying \Cref{thm:rational-asympt-normal} using codes in \cite{genus-poly-h-family}, we have the following asymptotic normality.

\begin{thm} \label{thm:C4-normality}
  Let $X_n$ be the genus of a random embedding of $C^4_n$, then $X_n$ is asymptotically normal, with mean $\mathfrak{m}n + O(1)$ and variance $\mathfrak{v}n + O(1)$, where
  \[
    \mathfrak{m} = \frac{4143}{3424}, \quad \mathfrak{v} = \frac{3706846585}{30106656768}.
  \]
\end{thm}

After some further modification, our implementation of \cref{thm:H-linear-rational} can also be applied to some other families of graphs. The $m \times n$ grid graph is the Cartesian product of two paths, $P_m$ and $P_n$. Grid graphs form a natural and fundamental class of graphs in graph theory. A key result in graph theory is the Grid Minor Theorem by Robertson and Seymour \cite{RP86}. It states that there exists a function $f$ such that, for every positive integer $k$, $f(k)$ is the minimum integer ensuring that any graph with treewidth at least $f(k)$ contains a $k \times k$ grid as minor. In a personal communication \cite{Pri07}, Stahl inquired about the calculation of the genus polynomial for an $m \times n$ grid. In \cite{grid-genus}, Khan, Poshni and Gross were interested in the genus distribution of rectangular grids, and obtained a linear recurrence for the case of $3 \times n$ grid after a face-tracing argument with refined case analysis. While the same approach may be extended to $4 \times n$ grids, the needed case analysis will surely becomes extremely tedious and prone to error. An automatic approach is thus necessary.

Although not directly an $H$-linear family, the family of grids with fixed height is close to one, which motivates the following definition. Given an $H$-linear family $\graphseq = \graphseq(H, \varphi) = (G_n)_{n \geq 1}$ and another graph $H'$, we define a \tdef{capped $H$-linear} family $\graphseq' = (G'_n)_{n \geq 1}$ to be the graph sequence where $G'_n = G_n \vertamal_{\varphi'} H'$, where $\varphi'$ sends a non-empty subset of unmatched gluing vertices in $G_n$ to a subset of vertices in $H'$. It is clear that, for any fixed $k$, the family of $k \times n$ grids is a capped $H$-linear family. From the proof of \cref{thm:H-linear-rational}, we see that there is a linear operator $\mathbf{V}_{H'}$, induced by the vertex amalgamation by $H'$, that sends the face element of $G_n$ to that of $G'_n$ for all $n \geq 1$. It is thus clear that the genus generating function for $\graphseq'$ is also rational, and its denominator divides that of the corresponding $H$-linear family $\graphseq$. We may thus extend our implementation of \cref{thm:H-linear-rational} to capped $H$-linear families. With such an implementation, we obtain the following expression of the genus generating function $F_{3 \times n}$ for grids of size $3 \times n$:
\begin{align*}
    F_{3 \times n}(t, x) &= \frac{2tA(t, x)}{B(t, x)}, \\
    A(t, x) &= ((1728 x + 1728) x^4 t^3 - (864 x^2 + 1080 x + 72) x^2 t^2 - (252 x^2 + 126 x - 42) x t \\
    &\qquad + 18 x^2 + 29 x + 1), \\
    B(t, x) &= (1 - (30 x + 1) t + (168 x - 42) x t^2+ (1008 x + 72) x^2 t^3 - 1728 x^4 t^4).
\end{align*}
Its first terms agree with those given in \cite{grid-genus}, and the computation took a few seconds. Applying \Cref{thm:rational-asympt-normal}, we have the corresponding asymptotic normality result.
\begin{thm} \label{thm:grid3-normality}
  Let $X_n$ be the genus of a random embedding of $F_{3 \times n}$, then $X_n$ is asymptotically normal, with mean $\mathfrak{m}n + O(1)$ and variance $\mathfrak{v}n + O(1)$, where
  \[
    \mathfrak{m} = \frac{34}{41}, \quad \mathfrak{v} = \frac{4816}{68921}.
  \]
\end{thm}

We also computed the genus generating function $F_{4 \times n}$ for grids of size $4 \times n$, which took about 20 minutes on an ordinary laptop. Its denominator is
\begingroup
\allowdisplaybreaks
\begin{align*}
  &\quad (4210380767232 x^2 - 14179076407296 x + 3281620303872) x^{17} t^{13} \\ 
  &+ (4045267795968 x^4 - 6209795653632 x^3 - 18664430469120 x^2 - 1415499743232 x \\
  &\qquad + 21284093952) x^{14} t^{12} \\
  &+ (17902216249344 x^4 + 20408316788736 x^3 + 2597405958144 x^2 - 171514920960 x^1 \\
  &\qquad - 8342839296) x^{12} t^{11} \\
  &- (1674324836352 x^5 + 2842878726144 x^4 + 760514946048 x^3 - 244951921152 x^2 \\
  &\qquad - 25993322496 x + 350148096) x^{10} t^{10} \\
  &- (1728919323648 x^5 + 352823751936 x^4 + 93791131776 x^3 + 18348738048 x^2 - 444531456 x \\
  &\qquad - 31045248) x^8 t^9 \\
  &+ (196671587328 x^6 + 73040693760 x^5 - 7330553568 x^4 - 531265920 x^3 - 40426944 x^2 \\\
  &\qquad - 37219776 x + 29376) x^6 t^8 \\
  &+ (21002582016 x^5 + 6971414832 x^4 + 959759928 x^3 + 45284456 x^2 + 836632 x - 81600) x^5 t^7 \\
  &- (2832558336 x^5 + 73546128 x^4 + 48384420 x^3 - 9095532 x^2 - 1460732 x - 50860) x^4 t^6 \\
  &- (112888368 x^4 + 28685604 x^3 + 1997650 x^2 + 57846 x + 3118) x^3 t^5 \\
  &+ (9431280 x^4 + 1241028 x^3 - 230871 x^2 - 25617 x - 552) x^2 t^4 \\
  &+ (189996 x^3 + 21832 x^2 + 1995 x + 42) x t^3 - (10080 x^2 - 1098 x + 7) x t^2 - (94 x + 1) t + 1.
\end{align*}
\endgroup
Applying \Cref{thm:rational-asympt-normal}, with computation details given in \cite{genus-poly-h-family}, we have the corresponding asymptotic normality result.
\begin{thm} \label{thm:grid4-normality}
  Let $X_n$ be the genus of a random embedding of $F_{4 \times n}$, then $X_n$ is asymptotically normal, with mean $\mathfrak{m}n + O(1)$ and variance $\mathfrak{v}n + O(1)$, where
  \begin{align*}
    \mathfrak{m} &= \frac{16459489073106163}{12617003380174210} \approx 1.30455, \\[0.5em]
    \mathfrak{v} &= \frac{263243843228068098124475290497912538382647423379}{3012727955054903571167749061007339758242364191500} \approx 0.0873772.
  \end{align*}
\end{thm}

The related family of circular grids, which are obtained by identifying the first and the last column of vertices in a rectangular grid, is an $H$-circular family, whose genus generating function can also be computed automatically by our software, with the same denominator as the corresponding family of rectangular grids.

%We also note that, both \Cref{thm:H-linear-rational} and \Cref{thm:H-linear-rational-euler} are under the framework laid down in \cite[Theorem~IX.9]{flajolet}.

For any given $H$, to show that its related $H$-linear and $H$-circular families have asymptotically normal embedding distributions, we can compute their rational generating functions by \Cref{thm:H-linear-rational} or \Cref{thm:H-linear-rational-euler} and try to check whether conditions in \Cref{thm:rational-asympt-normal} apply. However, to show asymptotic normality for embedding genus in general seems elusive. First, for some $H$ with particular gluing, such as the case where $H$ is a tree and the gluing consists of only a pair of vertices, it is clear that both the $H$-linear and the $H$-circular families are planar, thus without asymptotic normality. Second, although \Cref{thm:H-linear-rational,thm:H-linear-rational-euler} fall under the framework of perturbation of linear system in \cite[Theorem~IX.10]{flajolet}, which leads naturally to asymptotic normality results, it needs the linear operator $\mathbf{T}_H|_{x=1}$ to be primitive, which is not the case in general. Indeed, even the more refined $\mathbf{T}_H$ encodes more redundant information than the matrices in \Cref{sec:double-edge-cycle}, which are themselves not primitive in general from \Cref{prob:stochastic-matrix-primitive} and the discussions above. A consequence of $\mathbf{T}_H$ not being primitive is the existence of poles with multiplicity, which is also observed in the computation of the genus generating function of triple-edge cycle graphs above. In this case, it is possible that the dominant singularity of the rational generating function is not a simple pole, for which \Cref{thm:rational-asympt-normal} does not apply. Fortunately, it never happens for our examples. For the Euler-genus generating function of double-edge cycle graphs studied in \Cref{sec:double-edge-euler-genus}, there is a cancellation, and the reduced rational fraction still satisfies conditions in \Cref{thm:rational-asympt-normal}. For the genus generating function of triple-edge cycle graphs, there are poles with multiplicities for $B(t, 1)$, but the dominant singularity is still simple, thus \Cref{thm:rational-asympt-normal} still applies. In the case where the dominant singularity is not a simple pole, it may still be possible to get the asymptotic normality of embedding distributions using \cite[Theorem~1]{Ben73}, which is stronger than \Cref{thm:rational-asympt-normal} and can sometimes deal with higher-order poles. However, such possibility shows the difficulties of turning our approach into a general analysis, instead of a method that may work case by case.

Albeit the difficulties, we propose the following conjecture of a sufficient condition for the genus distribution of $H$-linear and $H$-circular families to be asymptotically normal. Let $H$ be a graph with a self-gluing $\varphi : U_1 \to U_2$, and $K$ a minimal cut between $U_1$ and $U_2$, that is, a set of edges whose removal disconnects vertices in $U_1$ and those in $U_2$ and which is minimal by inclusion. A \tdef{swapping} $\swapping(H, \varphi, K)$ of $H$ with respect to $K$ is obtained as follows:
\begin{enumerate}
\item We cut each edge in $E$ into two, creating two new vertices, one in a connected component with some vertices in $U_1$, the other with $U_2$, and these are the only cases by the minimality of the cut. Let $N_1$ (resp. $N_2$) be the set of these new vertices connected to $U_1$ (resp. $U_2$). We note that there is a natural bijection $\psi : N_1, N_2$ that pairs the new vertices coming from the same edge.
\item We perform the gluing by identifying vertices in $U_1$ to those in $U_2$ according to $\varphi$. Let $H'$ be the final graph obtained, then $\swapping(H, \varphi, K)$ is the pair $(H', \psi)$.
\end{enumerate}

\begin{conj} \label{conj:H-family-normal-asympt}
  Let $H$ be a graph with a self-gluing $\varphi: U_1 \to U_2$. Suppose that there is a valid minimal cut $K$ such that $(H', \psi) = \swapping(H, \varphi, K)$ has two cycles $C_1, C_2$ sharing at least one vertex. Then the genus distributions of the $H$-linear family $\graphseq(H, \varphi)$ and the $H$-circular family $\graphseq^\circ(H, \varphi)$ are both asymptotically normal.
\end{conj}

The main obstacle in proving \Cref{conj:H-family-normal-asympt} is exactly the existence of an irreducible Markov chain associated to the linear system deduced from $H$ and $\varphi$. More concretely, we know that the genus generating function $F(t, x)$ can be written as $A(t,x) / B(t,x)$, and to apply \Cref{thm:rational-asympt-normal}, we need to check that the simple pole $t = \rho$ of $F(t, 1)$ is not a root of $A(t, 1)$, which needs a precise understanding of $A(t, x)$, which in turns comes from the interaction between the first element $\mathbb{A}_1$ and the operator $\mathbf{T}_H$ in the proof of \Cref{thm:H-linear-rational}. However, for the general case, it seems to be hard to extract the information we need.

\paragraph{Acknowledgment} We would like to thank the anonymous reviewers for their precious comments pointing out errors and imprecision, helping to improve the presentation of this article.
\printbibliography

\end{document}